\documentclass[11pt]{article}
\usepackage{amsmath,amssymb,amsthm,amscd}
\usepackage[all]{xy}
\usepackage{tikz}
\newtheorem{theorem}{Theorem}[section]
\newtheorem{lemma}[theorem]{Lemma}
\newtheorem{proposition}[theorem]{Proposition}
\newtheorem{corollary}[theorem]{Corollary}

\theoremstyle{definition}
\newtheorem{definition}[theorem]{Definition}
\newtheorem{setting}[theorem]{Setting}
\newtheorem{example}[theorem]{Example}
\newtheorem{remark}[theorem]{Remark}

\numberwithin{equation}{section}

\renewcommand{\phi}{\varphi}

\newcommand{\ep}{\varepsilon}

\newcommand{\Coker}{\operatorname{Coker}}
\renewcommand{\d}{\operatorname{d}}
\newcommand{\dH}{\operatorname{d_{\textrm{H}}}}
\newcommand{\diam}{\operatorname{diam}}

\renewcommand{\Im}{\operatorname{Im}}

\newcommand{\Ker}{\operatorname{Ker}}

\newcommand{\N}{\mathbb{N}}
\newcommand{\Z}{\mathbb{Z}}

\newcommand{\R}{\mathbb{R}}

\newcommand{\G}{\mathcal{G}}
\renewcommand{\H}{\mathcal{H}}
\newcommand{\V}{\mathcal{V}}
\newcommand{\E}{\mathcal{E}}
\newcommand{\W}{\mathcal{W}}
\newcommand{\F}{\mathcal{F}}

\title{Long exact sequences of homology groups of \'etale groupoids}
\author{Hiroki Matui \\
Graduate School of Science \\
Chiba University \\
Inage-ku, Chiba 263-8522, Japan
}
\date{}

\begin{document}
\maketitle

\begin{abstract}
When a pair of \'etale groupoids $\mathcal{G}$ and $\mathcal{G}'$ 
on totally disconnected spaces are related in some way, 
we discuss the difference of their homology groups. 
More specifically, we treat two basic situations. 
In the subgroupoid situation, $\mathcal{G}'$ is assumed to be 
an open regular subgroupoid of $\mathcal{G}$. 
In the factor groupoid situation, we assume that 
$\mathcal{G}'$ is a quotient of $\mathcal{G}$ and 
the factor map $\mathcal{G}\to\mathcal{G}'$ is proper and regular. 
For each, we show that 
there exists a long exact sequence of homology groups. 
We present examples 
which arise from SFT groupoids and hyperplane groupoids. 
\end{abstract}

%%%%%%%%%%%%%%%%%%%%%%%%%%%%%%%%%%%%%%%%%%%%%%%%%%%%%%%%%%%%
\section{Introduction}
%%%%%%%%%%%%%%%%%%%%%%%%%%%%%%%%%%%%%%%%%%%%%%%%%%%%%%%%%%%%

We investigate homology groups of \'etale groupoids 
on totally disconnected spaces. 
\'Etale groupoids (called $r$-discrete groupoids in \cite{Re_text}) 
provide us a natural framework for unified treatment of 
various topological dynamical systems. 
The homology groups of \'etale groupoids were first 
introduced and studied by M. Crainic and I. Moerdijk \cite{CM00crelle}. 
Later, in \cite{Ma12PLMS,Ma15crelle,Ma16Adv}, 
when the underlying space is totally disconnected, 
the homology groups were computed for some specific examples, 
and the connection with topological full groups 
as well as $K$-theory of $C^*$-algebras was discussed. 
See \cite{FKPS19Munster,Or20JNG,Sc20ETDS,Yi20BullAust,
PY2006arXiv,BDGW2104arXiv} 
for further developments. 

The goal of this paper is the computation of homology groups 
of \'etale groupoids on totally disconnected spaces. 
To be more specific, 
we will be concerned with a pair of \'etale groupoids 
which are related in some way, 
and look at the difference of homology groups of the two groupoids. 
Loosely speaking, we may call it the relative homology theory. 
Analogous results for $K$-theory of $C^*$-algebras have been already 
obtained in \cite{Pu97JOP,Pu98Rocky,Pu21Munster} by I. F. Putnam. 
For a groupoid $\G$, 
one can associate its reduced groupoid $C^*$-algebra $C^*_r(\G)$. 
Given a pair of groupoids $\G$ and $\G'$ satisfying certain conditions, 
we get an inclusion $C^*_r(\G')\hookrightarrow C^*_r(\G)$ 
of $C^*$-algebras. 
By using his excision theorem, Putnam established a method of 
computing the relative $K$-groups $K_i(C^*_r(\G');C^*_r(\G))$. 
More precisely, 
he constructed another pair of groupoids $\H$ and $\H'$ 
which are tractable compared with $\G$ and $\G'$, 
and then showed that $K_i(C^*_r(\G');C^*_r(\G))$ are 
isomorphic to $K_i(C^*_r(\H');C^*_r(\H))$. 
In the present paper, 
we discuss the homology theoretic counterpart. 
Namely, we will prove that the two pairs of groupoids 
have the isomorphic relative homology groups 
(Proposition \ref{iso_sub} and \ref{iso_factor}). 
As a consequence, via this isomorphism, 
the long exact sequence coming from $\G$ and $\G'$ is 
linked to that coming from $\H$ and $\H'$. 
See Theorem \ref{LES_sub} and \ref{LES_factor} 
for the precise statement. 

In the same way as Putnam's results, 
we deal with two situations: 
one is the subgroupoid situation and 
the other is the factor groupoid situation. 
By using simple examples coming from Cantor minimal systems 
(minimal $\Z$-actions on Cantor sets), 
we would like to explain the basic idea behind the proof 
in the two situations. 
Let $\phi:\Z\curvearrowright X$ be a Cantor minimal system and 
consider the transformation groupoid $\G=\Z\ltimes X$. 
It is well-known that 
\[
H_0(\G)\cong C(X,\Z)/\{f-f\circ\phi_1\mid f\in C(X,\Z)\},\quad 
H_1(\G)=\Z
\]
and $H_n(\G)=0$ for $n\geq2$. 
Fix a point $y\in X$ and consider the subgroupoid 
\[
\G_y:=\G\setminus\left\{(m,\phi_n(y))\mid 
(n{\leq}0\wedge m{+}n{>}0)\vee(n{>}0\wedge m{+}n{\leq}0)\right\}. 
\]
It is also known that $\G_y$ is a (so-called) AF groupoid, 
and $H_0(\G_y)=H_0(\G)$ and $H_n(\G_y)=0$ for $n\geq1$. 
The two groupoids $\G$ and $\G_y$ differ only on the orbit of $y$. 
So, let us look only at this single orbit. 
Then the restriction of $\G$ to the orbit is 
just a transitive action of $\Z$ on $\Z$, 
which yields the difference between $H_1(\G)=\Z$ and $H_1(\G_y)=0$. 
Next, let us turn to an example of factor groupoids. 
Suppose that there exists an asymptotic pair $x,x'\in X$, i.e. 
\[
\lim_{\lvert n\rvert\to\infty}\d(\phi_n(x),\phi_n(x'))=0. 
\]
By glueing the points $\phi_n(x),\phi_n(x')$ and 
identifying them with a single point for each $n$, 
we obtain another Cantor minimal system $(Y,\psi)$. 
There exists a canonical factor map $\pi:(X,\phi)\to(Y,\psi)$. 
It is known that $H_0(\G)$ is isomorphic to 
the direct sum of $H_0$ of $(Y,\psi)$ and $\Z$ 
(see \cite{GPS01MathScand} for details). 
Again, the two dynamical systems differ only on a single $\psi$-orbit, 
and the $\Z$ summand of $H_0(\G)$ arises from this orbit. 
In both cases (subgroupoid and factor groupoid), 
the difference of the homology groups of two groupoids can be 
deduced from the restriction of the $\Z$-action to a single orbit. 
The issue here is the topology of that single orbit. 
Although this is a subset of the given Cantor set, 
the use of the relative topology is not a good idea. 
Since every orbit is dense, 
the relative topology on it cannot be locally compact. 
The correct topology in this case is the discrete topology. 
In a more general setting, we need to seek the right topology 
on the restriction of the given \'etale groupoid 
so that it becomes an \'etale groupoid. 
The regularity assumptions 
(Definition \ref{regularinclusion} and Definition \ref{defofregular}) 
enables us to resolve this technical issue. 

We would like to also mention 
a homology theory for Smale spaces defined by Putnam \cite{Pu14Memoir}. 
For an irreducible Smale space $(X,\phi)$, 
he introduced the stable homology groups $H^s_*(X,\phi)$ and 
the unstable homology groups $H^u_*(X,\phi)$ 
(\cite[Definition 5.5.2]{Pu14Memoir}). 
Loosely speaking, these groups are computed 
via a certain factor map 
from a shift of finite type $(\Sigma,\sigma)$ to $(X,\phi)$. 
In the case that the stable sets are totally disconnected, 
we may arrange that the factor map is $s$-bijective, 
which induces a groupoid inclusion $\G'\hookrightarrow\G$ 
just like the one considered in this paper. 
Furthermore, it is known that 
$H^s_*(X,\phi)$ coincides with the homology groups $H_*(\G)$ 
(\cite[Theorem 4.1]{PY2104arXiv}). 
This shows that homology computations associated with inclusions 
have been considered before the present paper. 

The paper is organized as follows. 
In Section 2, we set up notation and 
review some basic facts on \'etale groupoids and their homology groups. 
In Section 3, we deal with the subgroupoid situation, 
i.e.\ the case that 
$\G'$ is a subgroupoid of $\G$ and 
the inclusion $\G'\subset\G$ is regular 
(Definition \ref{regularinclusion}). 
In this setting, 
the chain complex $\mathcal{C}(\G',A)$ is naturally 
identified with a subcomplex of $\mathcal{C}(\G,A)$. 
Section 4 is devoted to the factor groupoid situation, 
i.e.\ the case that 
$\pi:\G\to\G'$ is a regular proper homomorphism 
(Definition \ref{defofregular}). 
In this setting, 
there exists an injective homomorphism 
$\pi^*:\mathcal{C}(\G',A)\to\mathcal{C}(\G,A)$. 
First, we study regular maps between locally compact metric spaces. 
When the spaces are totally disconnected, 
we can introduce the notion of reduction maps 
(Definition \ref{defofreduction}, Proposition \ref{reductionisiso}). 
By using this reduction map, 
we will prove the isomorphism between relative homology groups. 
In Section 5, 
we give two examples from SFT groupoids. 
Applying Theorem \ref{LES_factor} and Theorem \ref{LES_sub}, 
one obtains long exact sequences. 
Section 6 deals with an example from hyperplane groupoids. 
We compute the homology groups of the octagonal tiling 
and the Penrose tiling by using our long exact sequences.

%%%%%%%%%%%%%%%%%%%%%%%%%%%%%%%%%%%%%%%%%%%%%%%%%%%%%%%%%%%%
\section{Preliminaries}
%%%%%%%%%%%%%%%%%%%%%%%%%%%%%%%%%%%%%%%%%%%%%%%%%%%%%%%%%%%%

In this section, 
we present the basic definitions and well-known results 
concerning \'etale groupoids and their homology groups.

%%%%%%%%%%%%%%%%%%%%%%%%%%%%%%%%%%%%%%%%%%%%%%%%%%%%%%%%%%%%
\subsection{\'Etale groupoids}
%%%%%%%%%%%%%%%%%%%%%%%%%%%%%%%%%%%%%%%%%%%%%%%%%%%%%%%%%%%%

The cardinality of a set $A$ is written $\#A$ and 
the characteristic function of $A$ is written $1_A$. 
We say that a subset of a topological space is clopen 
if it is both closed and open. 
A topological space is said to be totally disconnected 
if its connected components are singletons. 
By a Cantor set, 
we mean a compact, metrizable, totally disconnected space 
with no isolated points. 
It is known that any two such spaces are homeomorphic. 

In this article, by an \'etale groupoid 
we mean a second countable locally compact Hausdorff groupoid 
such that the range map is a local homeomorphism. 
(We emphasize that our \'etale groupoids are always assumed to be Hausdorff, 
while non-Hausdorff groupoids are also studied actively.) 
We refer the reader to \cite{Re_text,R08Irish} 
for background material on \'etale groupoids. 
For an \'etale groupoid $\G$, 
we let $\G^{(0)}$ denote the unit space and 
let $s$ and $r$ denote the source and range maps, 
i.e.\ $s(g)=g^{-1}g$, $r(g)=gg^{-1}$. 
For a subset $Y\subset\G^{(0)}$, 
the reduction of $\G$ to $Y$ is $r^{-1}(Y)\cap s^{-1}(Y)$ and 
denoted by $\G|Y$. 
If $Y$ is open, then 
the reduction $\G|Y$ is an \'etale subgroupoid of $\G$ in an obvious way. 

A subset $U\subset\G$ is called a $\G$-bisection 
if $r|U$ and $s|U$ are injective. 
For $U,V\subset\G$, we set 
\[
U^{-1}:=\{g\in\G\mid g^{-1}\in U\}
\]
and 
\[
UV:=\{gg'\in\G\mid g\in U,\ g'\in V,\ s(g)=r(g')\}. 
\]
When $U,V$ are $\G$-bisections, 
$U^{-1}$ and $UV$ are also $\G$-bisections. 

A subset $X\subset\G^{(0)}$ is said to be full 
if $r(s^{-1}(X))=\G^{(0)}$ holds. 
The \'etale groupoids $\G$ and $\H$ are said to be Kakutani equivalent 
if there exists full clopen subsets 
$X\subset\G^{(0)}$ and $Y\subset\H^{(0)}$ 
such that $\G|X\cong\H|Y$ 
(\cite[Definition 4.1]{Ma12PLMS},\cite[Definition 3.8]{FKPS19Munster}).

%%%%%%%%%%%%%%%%%%%%%%%%%%%%%%%%%%%%%%%%%%%%%%%%%%%%%%%%%%%%
\subsection{Homology groups}
%%%%%%%%%%%%%%%%%%%%%%%%%%%%%%%%%%%%%%%%%%%%%%%%%%%%%%%%%%%%

Let $A$ be a topological abelian group. 
For a locally compact Hausdorff space $X$, we denote by $C_c(X,A)$ 
the set of $A$-valued continuous functions with compact support. 
When $X$ is compact, we simply write $C(X,A)$. 
With pointwise addition, $C_c(X,A)$ is an abelian group. 
Let $\pi:X\to Y$ be a local homeomorphism 
between locally compact Hausdorff spaces. 
For $f\in C_c(X,A)$, we define a map $\pi_*(f):Y\to A$ by 
\[
\pi_*(f)(y):=\sum_{\pi(x)=y}f(x). 
\]
It is not so hard to see that $\pi_*(f)$ belongs to $C_c(Y,A)$ and 
that $\pi_*$ is a homomorphism from $C_c(X,A)$ to $C_c(Y,A)$. 
Besides, if $\pi':Y\to Z$ is another local homeomorphism to 
a locally compact Hausdorff space $Z$, then 
one can check $(\pi'\circ\pi)_*=\pi'_*\circ\pi_*$ in a direct way. 
Thus, $C_c(\cdot,A)$ is a covariant functor 
from the category of locally compact Hausdorff spaces 
with local homeomorphisms 
to the category of abelian groups with homomorphisms. 

Let $\G$ be an \'etale groupoid. 
For $n\in\N$, we write $\G^{(n)}$ 
for the space of composable strings of $n$ elements in $\G$, that is, 
\[
\G^{(n)}:=\{(g_1,g_2,\dots,g_n)\in\G^n\mid
s(g_i)=r(g_{i+1})\text{ for all }i=1,2,\dots,n{-}1\}. 
\]
For $i=0,1,\dots,n$, 
we let $d_i:G^{(n)}\to G^{(n-1)}$ be a map defined by 
\[
d_i(g_1,g_2,\dots,g_n):=\begin{cases}
(g_2,g_3,\dots,g_n) & i=0 \\
(g_1,\dots,g_ig_{i+1},\dots,g_n) & 1\leq i\leq n{-}1 \\
(g_1,g_2,\dots,g_{n-1}) & i=n. 
\end{cases}
\]
When $n=1$, we let $d_0,d_1:\G^{(1)}\to\G^{(0)}$ be 
the source map and the range map, respectively. 
For $j=0,1,\dots,n$, 
we let $s_j:G^{(n)}\to G^{(n+1)}$ be a map defined by 
\[
s_j(g_1,g_2,\dots,g_n):=\begin{cases}
(r(g_1),g_1,g_2,\dots,g_n) & j=0 \\
(g_1,\dots,g_j,r(g_{j+1}),g_{j+1},\dots,g_n) & 1\leq j\leq n{-}1 \\
(g_1,g_2,\dots,g_n,s(g_n)) & j=n. 
\end{cases}
\]
Clearly the maps $d_i$ and $s_j$ are local homeomorphisms. 
The spaces $(\G^{(n)})_{n\geq0}$ 
together with the face maps $d_i$ and the degeneracy maps $s_j$ 
form a simplicial space, 
which is called the nerve of $\G$ 
(see \cite[Example I.1.4]{GJ_text} for instance). 
Applying the covariant functor $C_c(\cdot,A)$ to the nerve of $\G$, 
we obtain the simplicial abelian group $(C_c(\G^{(n)},A))_{n\geq0}$. 
Define the homomorphisms $\partial_n:C_c(\G^{(n)},A)\to C_c(\G^{(n-1)},A)$ 
by 
\[
\partial_n:=\sum_{i=0}^n(-1)^id_{i*}. 
\]
The abelian groups $C_c(\G^{(n)},A)$ 
together with the boundary operators $\partial_n$ form a chain complex 
$\mathcal{C}(\G,A):=(C_c(\G^{(n)},A))_n$, 
which is called the Moore complex of the simplicial abelian group 
(see \cite[Chapter III.2]{GJ_text} for instance). 

\begin{definition}
[{\cite[Section 3.1]{CM00crelle},\cite[Definition 3.1]{Ma12PLMS}}]
\label{homology}
For $n\geq0$, we let $H_n(\G,A)$ be the homology groups of 
the Moore complex $\mathcal{C}(\G,A)$, 
i.e.\ $H_n(\G,A):=\Ker\partial_n/\Im\partial_{n+1}$, 
and call them the homology groups of $\G$ with constant coefficients $A$. 
When $A=\Z$, we simply write $H_n(\G):=H_n(\G,\Z)$. 
\end{definition}

It is known that if $\G$ and $\H$ are Kakutani equivalent, 
then they have the same homology groups, 
that is, $H_n(\G,A)\cong H_n(\H,A)$ 
(see \cite[Theorem 4.8]{Ma12PLMS} or \cite[Lemma 4.3]{FKPS19Munster}). 

Let $\G$ and $\H$ be \'etale groupoids and 
let $\pi:\G\to\H$ be a continuous and proper homomorphism. 
Define $\pi^{(n)}:\G^{(n)}\to\H^{(n)}$ by 
\[
\pi^{(n)}(g_1,g_2,\dots,g_n):=(\pi(g_1),\pi(g_2),\dots,\pi(g_n)). 
\]
It is easy to see that $\pi^{(n)}$ is also proper. 
Hence, we get homomorphisms from $C_c(\H^{(n)},A)$ to $C_c(\G^{(n)},A)$ 
by pull back, which commute with the boundary operators $\partial_n$. 
Thus, 
there exists a homomorphism $\pi^*:\mathcal{C}(\H,A)\to\mathcal{C}(\G,A)$.

%%%%%%%%%%%%%%%%%%%%%%%%%%%%%%%%%%%%%%%%%%%%%%%%%%%%%%%%%%%%
\section{Subgroupoid situation}
%%%%%%%%%%%%%%%%%%%%%%%%%%%%%%%%%%%%%%%%%%%%%%%%%%%%%%%%%%%%

In this section, 
we discuss a long exact sequence of homology groups 
of \'etale groupoids arising from the subgroupoid situation. 
The precise setting is given as follows. 

\begin{setting}\label{sub}
Let $\G$ be an \'etale groupoid 
and let $\G'\subset\G$ be an open subgroupoid with $\G'^{(0)}=\G^{(0)}$. 
Assume that $\G$ is totally disconnected. 
Set $\Delta:=\G\setminus\G'$. 
We define groupoids $\H$ and $\H'$ by 
\[
\H:=\G|r(\Delta)\quad\text{and}\quad \H':=\G'|r(\Delta). 
\]
Clearly $\H'$ is a subgroupoid of $\H$, $\H=\H'\cup\Delta$ 
and $\H^{(0)}=\H'^{(0)}=r(\Delta)$. 
\end{setting}

We would like to put suitable topologies on $\H$ and $\H'$ 
so that they become \'etale groupoids. 
To this end, 
we need certain technical conditions for the inclusion $\G'\subset\G$. 

\begin{definition}[{\cite[Definition 6.3]{Pu21Munster}}]
\label{regularinclusion}
Let $\G'\subset\G$ be as in Setting \ref{sub}. 
We say that the inclusion $\G'\subset\G$ is regular 
if the map $r:\Delta\to r(\Delta)$ is open, 
when the image is given the final topology with respect to $r$ 
(i.e. $U\subset r(\Delta)$ is open if and only if 
$r^{-1}(U)\cap\Delta$ is open in $\Delta$). 
\end{definition}

\begin{remark}\label{openness}
It is easy to see the following conditions are equivalent. 
\begin{enumerate}
\item The map $r:\Delta\to r(\Delta)$ is open, 
when the image is given the quotient topology. 
\item For any open subset $U\subset\Delta$ of $\Delta$, 
$r^{-1}(r(U))\cap\Delta$ is open in $\Delta$. 
\item For any open subset $U\subset\Delta$ and 
a sequence $(g_k)_k$ in $\Delta$, 
if $(g_k)_k$ converges to $g\in\Delta$ and $r(g)$ is in $r(U)$, 
then $r(g_k)$ is in $r(U)$ for all sufficiently large $k$. 
\end{enumerate}
\end{remark}

\begin{remark}
In \cite[Section 6]{Pu21Munster}, 
the groupoid $\G$ is not assumed to be \'etale or totally disconnected. 
Thus, the results of \cite[Section 6]{Pu21Munster} can be applied to 
actions of continuous groups on connected spaces. 
However, some of them require that 
the inclusion $\G'\subset\G$ is of finite index 
(see \cite[Definition 6.11]{Pu21Munster}). 
In this article, we do not need to assume this condition. 
\end{remark}

In what follows, 
we denote the multiplication map on $\G$ by $m:\G^{(2)}\to\G$. 
By \cite[Lemma 6.6 (1)]{Pu21Munster}, the map 
\[
m:(\Delta\times\Delta)\cap m^{-1}(\G')\to\H'
\]
is surjective. 

\begin{lemma}\label{misopen}
Let $\G'\subset\G$ and $\H'\subset\H$ be as in Setting \ref{sub}. 
Suppose that $\G'\subset\G$ is regular. 
Then the map $m:(\Delta\times\Delta)\cap m^{-1}(\G')\to\H'$ is open, 
when the image is given the quotient topology. 
\end{lemma}
\begin{proof}
Take open subsets $U_1,U_2$ of $\G$ and set 
\[
V:=(U_1\times U_2)\cap(\Delta\times\Delta)\cap m^{-1}(\G'). 
\]
Let $((g_k,h_k))_k$ be a sequence converging to $(g,h)$ 
in $(\Delta\times\Delta)\cap m^{-1}(\G')$. 
Assume $gh\in m(V)$. 
It suffices to show that $g_kh_k$ belongs to $m(V)$ 
for all sufficiently large $k$ (see Remark \ref{openness} (3)). 
There exists $(a,b)\in V$ such that $ab=gh$. 
Let $c:=bh^{-1}=a^{-1}g$. 
Since $r(h_k)\to r(h)=s(c)$, 
we can find a sequence $(c_k)_k$ in $\G$ such that 
$s(c_k)=r(h_k)$ and $c_k\to c$. 
Put $A:=\{k\in\N\mid c_k\in\G'\}$ and 
$B:=\{k\in\N\mid c_k\in\Delta\}=\N\setminus A$. 

When $A$ is an infinite set, 
we may think of $(c_k)_{k\in A}$ as a subsequence of $(c_k)_k$. 
For $k\in A$, $c_kh_k$ is in $\Delta$ and $c_kh_k\to ch=b$. 
Similarly we have $g_kc_k^{-1}$ is in $\Delta$ 
and $g_kc_k^{-1}\to gc^{-1}=a$. 
It follows that 
\[
g_kh_k=g_kc_k^{-1}c_kh_k=m(g_kc_k^{-1},c_kh_k)
\]
belongs to $m(V)$ for all sufficiently large $k\in A$. 

When $B$ is an infinite set, 
we may think of $(c_k)_{k\in B}$ as a subsequence of $(c_k)_k$. 
It follows from \cite[Lemma 6.4]{Pu21Munster} that 
$(\Delta\times\Delta)\cap m^{-1}(\Delta)$ is 
open in $\Delta\times\Delta$. 
This, together with $c_kh_k\to ch=b\in\Delta$, 
implies that $c_kh_k\in\Delta$ for all sufficiently large $k\in B$. 
In the same way, $g_kc_k^{-1}$ is in $\Delta$. 
Therefore, we obtain the same conclusion as above. 

Consequently, we can deduce that 
$g_kh_k$ belongs to $m(V)$ for all sufficiently large $k\in\N$, 
which completes the proof. 
\end{proof}

Now, we introduce topologies on $\H$ and $\H'$ 
in the following way. 

\begin{definition}[{\cite[Definition 6.5, Theorem 6.8]{Pu21Munster}}]
Let $\G'\subset\G$ and $\H'\subset\H$ be as in Setting \ref{sub}. 
\begin{enumerate}
\item Consider the surjection 
\[
m:(\Delta\times\Delta)\cap m^{-1}(\G')\to\H', 
\]
and endow $\H'$ with the quotient topology from this map. 
\item Endow $\H=\H'\cup\Delta$ with the disjoint union topology, 
that is, both $\H'$ and $\Delta$ are clopen. 
\end{enumerate}
\end{definition}

The following theorem says that if $\G'\subset\G$ is regular, 
then $\H$ and $\H'$ with the topologies above are \'etale groupoids 
as desired. 

\begin{theorem}\label{Hisetale}
Let $\G'\subset\G$ and $\H'\subset\H$ be as in Setting \ref{sub}. 
Suppose that $\G'\subset\G$ is regular. 
\begin{enumerate}
\item Both $\H$ and $\H'$ are \'etale groupoids 
\item Both $\H$ and $\H'$ are totally disconnected. 
\item $\H'$ is a clopen subgroupoid of $\H$. 
\end{enumerate}
\end{theorem}
\begin{proof}
(1)
By \cite[Theorem 6.7, 6.8]{Pu21Munster}, 
both $\H$ and $\H'$ are locally compact Hausdorff groupoids. 
(We remark that $\H'$ being locally compact immediately 
follows from Lemma \ref{misopen}.) 
By Lemma \ref{misopen}, 
$\H'$ is second countable, and so is $\H$. 
It remains for us to show that $\H$ and $\H'$ are \'etale. 
To this end, since $\G$ is \'etale and 
the topologies on $\H$ and $\H'$ are finer 
than the relative topologies from $\G$, 
it suffices to show that 
$r:\H'\to\H^{(0)}$ and $r:\Delta\to\H^{(0)}$ are open. 
When $U\subset\G$ is an open $\G$-bisection, 
$r(U\cap\Delta)=(U\cap\Delta)(U\cap\Delta)^{-1}$ is open 
thanks to Lemma \ref{misopen}. 
Therefore $r|\Delta$ is open. 
Let us consider the following commutative diagram: 
\[
\xymatrix@M=10pt{
(\Delta\times\Delta)\cap m^{-1}(\G') \ar[r]^-m \ar[d]_p 
& \H' \ar[d]^r \\
\Delta \ar[r]_r & \H^{(0)}, 
}
\]
where $p$ is defined by $p(g,h)=g$. 
As $p$ is clearly open, $r|\H'$ is also open. 

(2)
This is obvious, 
because the topologies on $\H$ and $\H'$ are finer 
than the relative topologies from $\G$, 

(3)
This is clear from the definition of the topology on $\H$. 
\end{proof}

Next, we would like to show that 
$\G^{(n)}\setminus\G'^{(n)}$ and $\H^{(n)}\setminus\H'^{(n)}$ are 
canonically homeomorphic, 
which implies that $\G'\subset\G$ and $\H'\subset\H$ share 
the same relative homology theory. 
We need the following technical lemma. 

\begin{lemma}
Let $\G'\subset\G$ and $\H'\subset\H$ be as in Setting \ref{sub}. 
Suppose that $\G'\subset\G$ is regular. 
Let $(g_k)_k$ be a sequence in $\H$ and let $g\in\H$. 
The following conditions are equivalent. 
\begin{enumerate}
\item $g_k\to g$ in $\H$. 
\item $r(g_k)\to r(g)$ in $\H$, $s(g_k)\to s(g)$ in $\H$ 
and $g_k\to g$ in $\G$. 
\item $r(g_k)\to r(g)$ in $\H$ and $g_k\to g$ in $\G$. 
\item $s(g_k)\to s(g)$ in $\H$ and $g_k\to g$ in $\G$. 
\end{enumerate}
\end{lemma}
\begin{proof}
The implications (1)$\implies$(2), (2)$\implies$(3) and 
(2)$\implies$(4) are obvious. 
It suffices to show (3)$\implies$(1). 

First, we consider the case $g\in\Delta$. 
Let $U\subset\G$ be an open $\G$-bisection such that $g\in U$. 
Since $(g_k)_k$ converges to $g$ in $\G$, 
$g_k$ is in $U$ eventually. 
We also have $r(g_k)\in r(U\cap\Delta)$ eventually, 
because $r(U\cap\Delta)$ is open in $\H$ and $r(g_k)\to r(g)$ in $\H$. 
It follows that $g_k$ belongs to $U\cap\Delta$ eventually, 
which means that $g_k$ converges to $g$ in $\H$. 

Next, we consider the case $g\in\H'$. 
Choose $a,b\in\Delta$ so that $g=ab$. 
Let $U$ be an open $\G$-bisection such that $a\in U$. 
We have $r(g_k)\in r(U\cap\Delta)$ eventually, 
because $r(U\cap\Delta)$ is open in $\H$ and $r(g_k)\to r(g)$ in $\H$. 
Hence, there exists $a_k\in U\cap\Delta$ such that $r(g_k)=r(a_k)$. 
As $U\cap\Delta$ is an open $\H$-bisection and 
$r(a_k)=r(g_k)\to r(g)=r(a)$ in $\H$, 
we have $a_k\to a$ in $\H$. 
Then, $a_k^{-1}g_k$ converges to $a^{-1}g=b\in\Delta$ in $\G$, 
and $r(a_k^{-1}g_k)=s(a_k)$ converges to $r(a^{-1}g)=s(a)$ in $\H$. 
It follows from the first half of the proof that 
$a_k^{-1}g_k$ converges to $b\in\Delta$ in $\H$. 
Thus, $g_k=a_ka_k^{-1}g_k$ converges to $g$ in $\H$. 
\end{proof}

\begin{proposition}
Let $\G'\subset\G$ and $\H'\subset\H$ be as in Setting \ref{sub}. 
Suppose that $\G'\subset\G$ is regular. 
For any $n\in\N\cup\{0\}$, 
the topological spaces $\G^{(n)}\setminus\G'^{(n)}$ 
and $\H^{(n)}\setminus\H'^{(n)}$ are naturally homeomorphic. 
\end{proposition}
\begin{proof}
When $n=0$, 
both $\G^{(0)}\setminus\G'^{(0)}$ and $\H^{(0)}\setminus\H'^{(0)}$ 
are empty. 
Let $n\in\N$. 
One has 
\begin{align*}
\G^{(n)}\setminus\G'^{(n)}
&=\{(g_1,g_2,\dots,g_n)\in\G^{(n)}
\mid\text{$\exists i$ such that $g_i\in\Delta$}\}\\
&=\{(g_1,g_2,\dots,g_n)\in\H^{(n)}
\mid\text{$\exists i$ such that $g_i\in\Delta$}\}\\
&=\H^{(n)}\setminus\H'^{(n)}
\end{align*}
as sets. 
By definition, it is evident that 
the topology of $\H^{(n)}\setminus\H'^{(n)}$ is finer than 
that of $\G^{(n)}\setminus\G'^{(n)}$. 
It remains for us to prove that the identity map 
from $\G^{(n)}\setminus\G'^{(n)}$ to $\H^{(n)}\setminus\H'^{(n)}$ 
is continuous. 
For each $i=1,2,\dots,n$, we define 
\[
C_i:=\{(g_1,g_2,\dots,g_n)\in\G^{(n)}\mid g_i\in\Delta\}. 
\]
Clearly, each $C_i$ is a closed subset of $\G^{(n)}\setminus\G'^{(n)}$ and 
the union of them is equal to $\G^{(n)}\setminus\G'^{(n)}$. 
Hence it suffices to show that the identity map on $C_i$ is continuous. 
Let $(\xi_k)_k$ be a sequence in $C_i$ 
which converges to $\xi=(g_1,g_2,\dots,g_n)\in C_i$ 
in $\G^{(n)}\setminus\G'^{(n)}$. 
Let us write $\xi_k=(g_{k,1},g_{k,2},\dots,g_{k,n})$. 
For each $j=1,2,\dots,n$, we have $g_{k,j}\to g_j$ in $\G$. 
Since $g_{k,i}$ and $g_i$ are in $\Delta$, we get $g_{k,i}\to g_i$ in $\H$. 
In particular, $r(g_{k,i})\to r(g_i)$ and $s(g_{k,i})\to s(g_i)$ in $\H$. 
If $i\neq1$, we get $s(g_{k,i-1})\to s(g_{i-1})$ in $\H$, and so, 
by virtue of the lemma above, one has $g_{k,i-1}\to g_{i-1}$ in $\H$. 
If $i\neq n$, we get $r(g_{k,i+1})\to r(g_{i+1})$ in $\H$, and so, 
by virtue of the lemma above, one has $g_{k,i+1}\to g_{i+1}$ in $\H$. 
Repeating this argument, we can conclude that 
$g_{k,j}$ converges to $g_j$ in $\H$ for every $j$. 
Therefore, $\xi_k$ converges to $\xi$ in $\H^{(n)}\setminus\H'^{(n)}$, 
as desired. 
\end{proof}

As a direct consequence of the proposition above, 
we obtain the following. 

\begin{proposition}\label{iso_sub}
Let $\G'\subset\G$ and $\H'\subset\H$ be as in Setting \ref{sub}. 
Suppose that $\G'\subset\G$ is regular. 
Let $A$ be a discrete abelian group. 
Then, we have the following diagram: 
\[
\xymatrix@M=10pt{
0 \ar[r] & \mathcal{C}(\G',A) \ar[r] & \mathcal{C}(\G,A) \ar[r] 
& \mathcal{C}(\G,A)/\mathcal{C}(\G',A) \ar[r] \ar@{=}[d] & 0 \\
0 \ar[r] & \mathcal{C}(\H',A) \ar[r] & \mathcal{C}(\H,A) \ar[r] 
& \mathcal{C}(\H,A)/\mathcal{C}(\H',A) \ar[r] & 0. 
}
\]
\end{proposition}
\begin{proof}
It is easy to see that 
$\mathcal{C}(\G',A)$ is a subcomplex of $\mathcal{C}(\G,A)$. 
Since $A$ is discrete, 
the quotient complex $\mathcal{C}(\G,A)/\mathcal{C}(\G',A)$ 
consists of the abelian groups $(C_c(\G^{(n)}\setminus\G'^{(n)},A))_n$. 
The same is true for groupoids $\H$ and $\H'$. 
The proposition above tells us that 
$C_c(\G^{(n)}\setminus\G'^{(n)},A)$ equals 
$C_c(\H^{(n)}\setminus\H'^{(n)},A)$ for all $n$, 
which completes the proof. 
\end{proof}

We let $H_*(\G/\G',A)$ (resp. $H_*(\H/\H',A)$) 
symbolically denote the homology groups of 
the chain complex $\mathcal{C}(\G,A)/\mathcal{C}(\G',A)$ 
(resp. $\mathcal{C}(\H,A)/\mathcal{C}(\H',A)$). 

\begin{theorem}\label{LES_sub}
Let $\G'\subset\G$ and $\H'\subset\H$ be as in Setting \ref{sub}. 
Suppose that $\G'\subset\G$ is regular. 
Let $A$ be a discrete abelian group. 
Then we have the following diagram: 
\[
\xymatrix@M=10pt{
\dots \ar[r] & H_n(\G',A) \ar[r] & H_n(\G,A) \ar[r] 
& H_n(\G/\G',A) \ar[r] \ar@{=}[d] & \dots \\
\dots \ar[r] & H_n(\H',A) \ar[r] & H_n(\H,A) \ar[r] 
& H_n(\H/\H',A) \ar[r] & \dots, 
% & H_{n-1}(\G',A) \ar[r]^{H_{n-1}^*(\pi)} & H_{n-1}(\G,A) \ar[r] 
% & H_{n-1}(\G/\G',A) \ar[r] \ar@{=}[d] & \dots \\
% & H_{n-1}(\G'_\pi,A) \ar[r]^{H_{n-1}^*(\pi)} & H_{n-1}(\G_\pi,A) \ar[r] 
% & H_{n-1}(\G_\pi/\G'_\pi,A) \ar[r] & \dots
}
\]
where the two horizontal sequences are exact. 
\end{theorem}

%%%%%%%%%%%%%%%%%%%%%%%%%%%%%%%%%%%%%%%%%%%%%%%%%%%%%%%%%%%%
\section{Factor groupoid situation}
%%%%%%%%%%%%%%%%%%%%%%%%%%%%%%%%%%%%%%%%%%%%%%%%%%%%%%%%%%%%

In this subsection, 
we discuss a long exact sequence of homology groups 
of \'etale groupoids arising from the factor groupoid situation. 
At first, apart from \'etale groupoids, 
we introduce the notion of regularity for continuous surjections 
between locally compact metric spaces 
and establish various properties in Section 4.1. 
Then we will apply them to the factor groupoid situation 
in Section 4.2.

%%%%%%%%%%%%%%%%%%%%%%%%%%%%%%%%%%%%%%%%%%%%%%%%%%%%%%%%%%%%
\subsection{Regular maps}
%%%%%%%%%%%%%%%%%%%%%%%%%%%%%%%%%%%%%%%%%%%%%%%%%%%%%%%%%%%%

In this subsection, we collect several facts 
about regular maps between locally compact metric spaces. 
Most of them are reformulation of 
arguments given in \cite[Section 7]{Pu21Munster}. 

Let $(X,\d)$ be a metric space. 
For a compact subset $A\subset X$, its diameter is defined by 
\[
\diam(A):=\sup\{\d(x,y)\mid x,y\in A\}. 
\]
The Hausdorff metric for compact subsets is defined by 
\[
\dH(A,B):=\max\left\{\sup_{x\in A}\inf_{y\in B}\d(x,y),\ 
\sup_{y\in B}\inf_{x\in A}\d(x,y)\right\}
\]
where $A,B\subset X$ are compact. 
It is easy to see that the inequality 
\[
\lvert\diam(A)-\diam(B)\rvert\leq2\dH(A,B)
\]
holds. 

The regularity of a proper surjection $\pi:X\to X'$ 
is defined as follows. 

\begin{definition}[{\cite[Definition 7.6, Definition 7.7]{Pu21Munster}}]
\label{defofregular}
Let $(X,\d)$ and $(X',\d')$ be locally compact metric spaces 
and let $\pi:X\to X'$ be a continuous proper surjection. 
\begin{enumerate}
\item We say that $\pi$ is regular 
if the following condition is satisfied: 
for every $x'\in X'$ and $\ep>0$ 
there exists an open neighborhood $U'$ of $x'$ 
such that for any $y'\in U'$ either 
\[
\dH(\pi^{-1}(x'),\pi^{-1}(y'))<\ep
\quad\text{or}\quad 
\diam\pi^{-1}(y')<\ep
\]
holds. 
\item Set 
\[
X'_\pi:=\{x'\in X'\mid\#\pi^{-1}(x')>1\}
\]
and $X_\pi:=\pi^{-1}(X_\pi')$. 
We endow $X'_\pi$ with the metric 
\[
\d'_\pi(x',y'):=\dH(\pi^{-1}(x'),\pi^{-1}(y'))
\quad\forall x',y'\in X'_\pi, 
\]
and $X_\pi$ with the metric 
\[
\d_\pi(x,y):=\d(x,y)+\d'_\pi(\pi(x),\pi(y))
\quad\forall x,y\in X_\pi. 
\]
\end{enumerate}
\end{definition}

Our first task is to show that 
the metric spaces $X_\pi$ and $X'_\pi$ are locally compact 
under the assumption that $\pi:X\to X'$ is regular. 

\begin{remark}
In \cite[Definition 7.7]{Pu21Munster} 
the regularity condition is required only for $x'\in X'_\pi$, 
and the definition above is seemingly stronger 
than \cite[Definition 7.7]{Pu21Munster}. 
However, the two definitions are the same, 
because for $x'\in X'\setminus X'_\pi$ 
we can always find an open neighborhood $U'$ of $x'$ 
such that $\diam\pi^{-1}(y')<\ep$ holds 
for any $y'\in U'$. 
\end{remark}

It is easy to see the following. 

\begin{lemma}[{\cite[Proposition 7.8]{Pu21Munster}}]
\label{finertop}
Let $(X,\d)$ and $(X',\d')$ be locally compact metric spaces 
and let $\pi:X\to X'$ be a continuous proper surjection. 
The inclusion maps $(X_\pi,\d_\pi)\hookrightarrow(X,\d)$ 
and $(X'_\pi,\d'_\pi)\hookrightarrow(X',\d')$ are continuous. 
\end{lemma}

The following is the crucial conclusion of the regularity. 

\begin{lemma}\label{Kdeltaiscpt}
Let $(X,\d)$ and $(X',\d')$ be locally compact metric spaces 
and let $\pi:X\to X'$ be a continuous proper surjection. 
Suppose that $\pi$ is regular. 
For any compact subset $K$ of $X'$ and $\delta>0$, the set 
\[
K_\delta:=K\cap\{x'\in X'_\pi\mid\diam\pi^{-1}(x')\geq\delta\}
\]
is compact in $X'_\pi$. 
\end{lemma}
\begin{proof}
Clearly $K_\delta$ is closed. 
Let $(x'_k)_k$ be any sequence in $K_\delta$. 
Since $K$ is compact in $X'$, passing to a subsequence, 
we may assume that $(x'_k)_k$ converges to $x'$ in $X'$. 
We have 
\[
\diam\pi^{-1}(x')
\geq\liminf_{k\to\infty}\diam\pi^{-1}(x'_k)\geq\delta. 
\]
In particular, $x'$ belongs to $X'_\pi$. 
It suffices to show that $(x'_k)_k$ converges to $x'$ in $X'_\pi$. 
Let $\ep>0$. Assume $\ep<\delta$. 
By the regularity of $\pi$, 
we may find an open neighborhood $U'$ of $x'$ 
such that for any $y'\in U'$ either 
\[
\dH(\pi^{-1}(x'),\pi^{-1}(y'))<\ep
\quad\text{or}\quad 
\diam\pi^{-1}(y')<\ep
\]
holds. 
When $k$ is sufficiently large, $x'_k$ is in $U'$. 
Since $\diam\pi^{-1}(x'_k)\geq\delta>\ep$, 
we get $\dH(\pi^{-1}(x'),\pi^{-1}(x'_k))<\ep$ 
for all sufficiently large $k$. 
Therefore $(x'_k)_k$ converges to $x'$ in $X'_\pi$. 
\end{proof}

The following proposition says that 
$X_\pi$ and $X'_\pi$ are locally compact when $\pi$ is regular. 
This is discussed in \cite[Theorem 7.9]{Pu21Munster} 
in the setting of groupoids. 
We include a proof here for the reader's convenience. 
This proposition will be used in the next subsection 
to show that the groupoids $\G_\pi$ and $\G'_\pi$ are \'etale 
(Proposition \ref{Gpiisetale}). 

\begin{proposition}[{\cite[Theorem 7.9]{Pu21Munster}}]
\label{consequofregular}
Let $(X,\d)$ and $(X',\d')$ be locally compact metric spaces 
and let $\pi:X\to X'$ be a continuous proper surjection. 
Suppose that $\pi$ is regular. 
Then $(X_\pi,\d_\pi)$ and $(X'_\pi,\d'_\pi)$ are 
locally compact metric spaces, 
and $\pi:(X_\pi,\d_\pi)\to(X'_\pi,\d'_\pi)$ is 
a continuous proper open map. 
\end{proposition}
\begin{proof}
First, we show that $X'_\pi$ is locally compact. 
Let $x'$ be in $X'_\pi$. 
Since $X'$ is locally compact, 
there exists a compact neighborhood $K$ of $x'$ in $X'$. 
Put $\delta:=\diam\pi^{-1}(x')$ and define 
\[
K_{\delta/2}
:=K\cap\{y'\in X'_\pi\mid\diam\pi^{-1}(y')\geq\delta/2\}. 
\]
By Lemma \ref{Kdeltaiscpt}, 
$K_{\delta/2}$ is compact in $X'_\pi$. 
Let us check that $x'$ is an interior point of $K_{\delta/2}$. 
By Lemma \ref{finertop}, there exists $\ep>0$ such that 
\[
\d'_\pi(x',y')<\ep\implies y'\in K
\]
holds for all $y'\in X'_\pi$. 
Let $\ep':=\min\{\ep,\delta/4\}$. 
Suppose that $y'\in X'_\pi$ satisfies $\d'_\pi(x',y')<\ep'$. 
Evidently $y'$ is in $K$, and 
\begin{align*}
\diam\pi^{-1}(y')
&\geq\diam\pi^{-1}(x')-2\dH(\pi^{-1}(x'),\pi^{-1}(y'))\\
&=\diam\pi^{-1}(x')-2\d'_\pi(x',y')\geq\delta/2. 
\end{align*}
Hence $y'$ is in $K_{\delta/2}$, 
and so $x'$ is an interior point of $K_{\delta/2}$. 

Next, we show that $\pi:X_\pi\to X'_\pi$ is proper. 
Let $K\subset X'_\pi$ be a compact subset and 
let $(x_k)_k$ be any sequence in $\pi^{-1}(K)$. 
By Lemma \ref{finertop}, $K$ is compact in $X'$, 
and hence $\pi^{-1}(K)$ is compact in $X$. 
It follows that, by passing to a subsequence, 
we may assume that $(x_k)_k$ converges to $x$ in $X$ 
and that $(\pi(x_k))_k$ converges to $\pi(x)$ in $X'_\pi$. 
From the definition of the metric $\d_\pi$, 
we can conclude that $(x_k)_k$ converges to $x$ in $X_\pi$. 
Thus, $\pi^{-1}(K)$ is a compact subset of $X_\pi$. 

By the definition of $\d_\pi$ and $\d'_\pi$, 
the map $\pi:X_\pi\to X'_\pi$ is clearly continuous. 
Let $x\in X_\pi$. 
There exists a compact neighborhood 
$K\subset X'_\pi$ of $\pi(x)$ in $X'_\pi$, 
because $X'_\pi$ is locally compact. 
Then $\pi^{-1}(K)$ is a compact neighborhood of $x$, 
which implies that $X_\pi$ is also locally compact. 

Finally, let us prove that $\pi:X_\pi\to X'_\pi$ is open. 
Let $V$ be an open subset of $X_\pi$ and let $x\in V$. 
We would like to prove that 
$\pi(x)$ is an interior point of $\pi(V)$ in $X'_\pi$. 
Find $\ep>0$ so that 
\[
\d_\pi(x,y)<\ep\implies y\in V
\]
holds for all $y\in X_\pi$. 
We show 
\[
\d'_\pi(\pi(x),z)<\ep/2\implies z\in\pi(V)
\]
holds for all $z\in X'_\pi$. 
If $\d'_\pi(\pi(x),z)<\ep/2$, 
then $\dH(\pi^{-1}(\pi(x)),\pi^{-1}(z))<\ep/2$. 
In particular, 
there exists $y\in\pi^{-1}(z)$ such that $\d(x,y)<\ep/2$. 
It follows that 
\[
\d_\pi(x,y)
=\d(x,y)+\d'_\pi(\pi(x),\pi(y))
=\d(x,y)+\d'_\pi(\pi(x),z)<\ep. 
\]
Hence $y$ is in $V$, and $z=\pi(y)$ is in $\pi(V)$. 
\end{proof}

\begin{remark}
The definition of the regularity apparently involves 
the metrics $\d$ and $\d'$. 
But, this dependence is not essential, 
and the regularity of $\pi$ is a topological property. 
Furthermore, the topologies on $X_\pi$ and $X'_\pi$ do not 
depend on the choice of $\d$ and $\d'$, neither. 
These follow from the fact that 
the metric topology induced by the Hausdorff distance 
coincides with the Vietoris topology 
on the family of non-empty compact subsets of the space. 
We refer the reader to 
\cite[Theorem 3.3]{MR42109} or \cite[Theorem 8]{MR1930656}. 
\end{remark}

Next, we show that regularity is preserved by local homeomorphisms. 
We will use this proposition to prove Lemma \ref{pi^nregular}. 

\begin{proposition}\label{regularVSlocalhomeo}
For $i=1,2$, 
let $(X_i,\d_i)$ and $(X'_i,\d'_i)$ be locally compact metric spaces 
and let $\pi_i:X_i\to X'_i$ be a continuous proper surjection. 
Let $\phi:X_1\to X_2$ and $\phi':X'_1\to X'_2$ be 
surjective local homeomorphisms 
such that $\pi_2\circ\phi=\phi'\circ\pi_1$. 
Assume further that 
the restriction $\phi:\pi_1^{-1}(x')\to\pi_2^{-1}(\phi'(x'))$ is bijective 
for every $x'\in X'_1$. 
Then the following are equivalent. 
\begin{enumerate}
\item $\pi_1$ is regular. 
\item $\pi_2$ is regular. 
\end{enumerate}
Moreover, when $\pi_1$ and $\pi_2$ are regular, 
both $\phi:(X_1)_{\pi_1}\to(X_2)_{\pi_2}$ and 
$\phi':(X'_1)_{\pi_1}\to(X'_2)_{\pi_2}$ are local homeomorphisms. 
\end{proposition}
\[
\xymatrix@M=10pt{
X_1 \ar[r]^\phi \ar[d]_{\pi_1} & X_2 \ar[d]^{\pi_2} \\
X'_1 \ar[r]^{\phi'} & X'_2
}
\]
\begin{proof}
We show that (1) implies (2). 
Take $x'\in X'_2$ and $\ep>0$. 
Since $\phi':X'_1\to X'_2$ is surjective, 
there exists $z'\in X'_1$ such that $\phi'(z')=x'$. 
Let $K'\subset X'_1$ be a compact neighborhood of $z'$. 
As $\pi_1$ is proper, $\pi_1^{-1}(K')$ is compact, 
and so the restriction of $\phi$ to it is uniformly continuous. 
Therefore we can find $\delta>0$ so that 
\[
\d_1(x,y)<\delta\implies\d_2(\phi(x),\phi(y))<\ep
\]
holds for all $x,y\in\pi_1^{-1}(K')$. 
By the regularity of $\pi_1$, 
there exists an open neighborhood $U'$ of $z'$ 
such that for any $w'\in U'$ either 
\[
\dH(\pi_1^{-1}(z'),\pi_1^{-1}(w'))<\delta
\quad\text{or}\quad 
\diam\pi_1^{-1}(w')<\delta
\]
holds, where $\dH$ and $\diam$ stand for 
the Hausdorff metric and the diameter with respect to $\d_1$. 
We may assume that $U'$ is smaller than $K'$. 
Then, $\phi'(U')$ is an open neighborhood of $x'$. 
Since $\phi:\pi_1^{-1}(w')\to\pi_2^{-1}(\phi'(w'))$ is surjective 
for every $w'\in X'_1$, 
we can see that for any $y'\in\phi'(U')$ either 
\[
\dH(\pi_2^{-1}(x'),\pi_2^{-1}(y'))<\ep
\quad\text{or}\quad 
\diam\pi_2^{-1}(y')<\ep
\]
holds, where $\dH$ and $\diam$ stand for 
the Hausdorff metric and the diameter with respect to $\d_2$. 
It follows that $\pi_2$ is regular. 

Let us prove the other implication. 
Take $x'\in X'_1$ and $\ep>0$. 
Since $\phi'$ is a local homeomorphism, 
there exists an open neighborhood $V'$ of $x'$ 
such that $\phi':V'\to\phi'(V')$ is a homeomorphism. 
We may assume that $V'$ has compact closure. 
Then, $\phi:\pi_1^{-1}(V')\to\pi_2^{-1}(\phi'(V'))$ 
is a homeomorphism, 
because $\phi:\pi_1^{-1}(y')\to\pi_2^{-1}(\phi'(y'))$ is bijective 
for every $y'\in X'_1$. 
By the uniform continuity, 
we can find $\delta>0$ so that 
\[
\d_2(\phi(x),\phi(y))<\delta\implies\d_1(x,y)<\ep
\]
holds for all $x,y\in\pi_1^{-1}(V')$. 
By the regularity of $\pi_2$, 
there exists an open neighborhood $U'$ of $\phi'(x')$ 
such that for any $w'\in U'$ either 
\[
\dH(\pi_2^{-1}(\phi'(x')),\pi_2^{-1}(w'))<\delta
\quad\text{or}\quad 
\diam\pi_2^{-1}(w')<\delta
\]
holds, where $\dH$ and $\diam$ stand for 
the Hausdorff metric and the diameter with respect to $\d_2$. 
We may assume that $U'$ is smaller than $\phi'(V')$. 
Then, for any $y'\in\phi'^{-1}(U')$ either 
\[
\dH(\phi(\pi_1^{-1}(x')),\phi(\pi_1^{-1}(y')))<\delta
\quad\text{or}\quad 
\diam\phi(\pi_1^{-1}(y'))<\delta
\]
holds. 
Hence, for any $y'\in\phi'^{-1}(U')$ either 
\[
\dH(\pi_1^{-1}(x'),\pi_1^{-1}(y'))<\ep
\quad\text{or}\quad 
\diam\pi_1^{-1}(y')<\ep
\]
holds, where $\dH$ and $\diam$ stand for 
the Hausdorff metric and the diameter with respect to $\d_1$. 
Therefore we can conclude that $\pi_2$ is regular. 

Finally, let us prove that 
$\phi $ restricted to $(X_1)_{\pi_1}$ and $(X'_1)_{\pi_1}$ 
are local homeomorphisms 
under the assumption that $\pi_1$ (and hence $\pi_2$) is regular. 
For $i=1,2$, 
we write $Y_i:=(X_i)_{\pi_i}$ and $Y'_i:=(X'_i)_{\pi_i}$. 
Let $x'\in Y'_1$ and $V'\subset X'_1$ be as above. 
The open sets $\pi_1^{-1}(V')$ and $\pi_2^{-1}(\phi'(V'))$ 
have compact closures and 
$\phi$ gives a homeomorphism between them. 
Suppose that 
a sequence $(y'_k)_k$ in $V'\cap Y'_1$ converges to 
$y'\in V'\cap Y'_1$ in $Y'_1$. 
Then 
\begin{align*}
\lim_k\d'_{\pi_1}(y'_k,y')=0
&\iff\lim_k\dH(\pi_1^{-1}(y'_k),\pi_1^{-1}(y'))=0\\
&\implies\lim_k\dH(\phi(\pi_1^{-1}(y'_k)),\phi(\pi_1^{-1}(y')))=0\\
&\iff\lim_k\dH(\pi_2^{-1}(\phi'(y'_k)),\pi_2^{-1}(\phi'(y')))=0\\
&\iff\lim_k\d'_{\pi_2}(\phi'(y'_k),\phi'(y'))=0, 
\end{align*}
which means that 
$\phi':V'\cap Y'_1\to\phi'(V')\cap Y'_2$ is continuous. 
In the same way, we can see that its inverse is continuous. 
Hence $\phi'$ gives a homeomorphism 
between $V'\cap Y'_1$ and $\phi'(V')\cap Y'_2$. 
Thus, $\phi'$ gives a local homeomorphism from $Y'_1$ to $Y'_2$. 
For each $i=1,2$, 
the metric $\d_{\pi_i}$ on $Y_i$ is 
the sum of $\d_i$ and $\d'_{\pi_i}$, 
and so one can easily check that 
$\phi:Y_1\to Y_2$ is also a local homeomorphism. 
\end{proof}

We now introduce the notion of reduction maps. 
In the rest of this subsection, we keep the following setting. 
Let $(X,\d)$ and $(X',\d')$ be locally compact metric spaces 
and let $\pi:X\to X'$ be a continuous proper surjection. 
Suppose that $\pi$ is regular. 
Furthermore, we assume that 
both $X$ and $X'$ are totally disconnected. 
It follows from Lemma \ref{finertop} that 
both $X_\pi$ and $X'_\pi$ are totally disconnected, too. 
Let $A$ be a discrete abelian group. 

\begin{lemma}\label{welldefofPi}
For any $f\in C_c(X,A)$, 
there exists a compact open subset $L'\subset X'_\pi$ and 
a continuous function $g:X'_\pi\to A$ such that 
\[
f(x)=g(\pi(x))\quad\forall x\in X_\pi\setminus\pi^{-1}(L'). 
\]
\end{lemma}
\begin{proof}
The support $K:=\{x\in X\mid f(x)\neq0\}$ of $f$ is 
compact and open in $X$. 
There exists $\delta>0$ such that 
\[
\d(x,y)<\delta\implies f(x)=f(y)
\]
holds for any $x,y\in X$. 
By Lemma \ref{Kdeltaiscpt}, 
\[
K':=\pi(K)\cap\{x'\in X'_\pi\mid\diam\pi^{-1}(x')\geq\delta\}
\]
is compact in $X'_\pi$. 
Let $L'\subset X'_\pi$ be a compact open set containing $K'$. 
If $x,y\in X_\pi\setminus\pi^{-1}(L')$ satisfy $\pi(x)=\pi(y)$, 
then by the definition of $L'$ we get $\d(x,y)<\delta$. 
Thus $f(x)=f(y)$. 
Hence there exists a continuous function $g:X'_\pi\to A$ such that 
\[
f(x)=g(\pi(x))\quad\forall x\in X_\pi\setminus\pi^{-1}(L'). 
\]
\end{proof}

For $f\in C_c(X,A)$ and a compact open subset $K'\subset X'_\pi$, 
we define $f_{K'}\in C_c(X_\pi,A)$ by 
\[
f_{K'}(x)
:=\begin{cases}f(x)&x\in\pi^{-1}(K')\\0&\text{otherwise.}\end{cases}
\]

\begin{corollary}
When the compact open subset $K'$ is large enough, 
the equivalence class of $f_{K'}$ in $C_c(X_\pi,A)/C_c(X'_\pi,A)$ is 
uniquely determined. 
\end{corollary}
\begin{proof}
Let $L'\subset X'_\pi$ be the compact open subset as in the lemma above. 
Then, for any compact open sets $K'_1,K'_2$ containing $L'$, 
one has $f_{K'_1}-f_{K'_2}\in C_c(X'_\pi,A)$, 
where $C_c(X'_\pi,A)$ is regarded as a subgroup of $C_c(X_\pi,A)$ 
via the map $\pi:X_\pi\to X'_\pi$. 
Thus, 
the equivalence class of $f_{K'}$ in $C_c(X_\pi,A)/C_c(X'_\pi,A)$ is 
uniquely determined by letting $K'$ large enough. 
\end{proof}

\begin{definition}\label{defofreduction}
We define the reduction map 
\[
\Pi:C_c(X,A)\to C_c(X_\pi,A)/C_c(X'_\pi,A)
\]
by $\Pi(f):=f_{K'}+C_c(X'_\pi,A)$, 
where $f_{K'}$ is a function described above. 
\end{definition}

\begin{proposition}\label{reductionisiso}
The reduction map $\Pi$ is a surjective homomorphism, 
and its kernel equals $C_c(X',A)$. 
\end{proposition}
\begin{proof}
It is clear that $\Pi$ is a homomorphism. 
We show that $\Pi$ is surjective. 
Let $C\subset X_\pi$ be a compact open set and let $a\in A$. 
We define $a1_C\in C_c(X_\pi,A)$ by 
\[
a1_C(x):=\begin{cases}a&x\in C\\0&\text{otherwise.}\end{cases}
\]
It suffices to prove that 
the equivalence class of $a1_C$ is contained in the image of $\Pi$. 
Note that $\pi:X_\pi\to X'_\pi$ is continuous, open and proper 
by Proposition \ref{consequofregular}. 
It follows that $\pi(C)\subset X'_\pi$ is compact and open and 
that $\pi^{-1}(\pi(C))\subset X'_\pi$ is compact and open, too. 
Since $X$ is totally disconnected, 
there exists a compact open subset $U\subset X$ such that 
\[
C\subset U\quad\text{and}\quad 
U\cap(\pi^{-1}(\pi(C))\setminus C)=\emptyset. 
\]
By applying Lemma \ref{welldefofPi} to $a1_U\in C_c(X,A)$, 
we obtain a compact open set $L'\subset X'_\pi$ and 
a continuous function $g:X'_\pi\to A$ such that 
\[
a1_U(x)=g(\pi(x))\quad\forall x\in X_\pi\setminus\pi^{-1}(L'). 
\]
By replacing $L'$ with $L'\cup\pi(C)$, 
we may assume $\pi(C)\subset L'$. 
Since $X'$ is totally disconnected, 
there exists a compact open subset $V'\subset X'$ such that 
\[
\pi(C)\cap V'=\emptyset\quad\text{and}\quad 
L'\setminus\pi(C)\subset V'. 
\]
Set $W:=U\setminus\pi^{-1}(V')\subset X$. 
We would like to show that $a1_W\in C_c(X,A)$ meets the requirement. 
Define a continuous function $\tilde g:X'_\pi\to A$ by 
\[
\tilde g(x')
:=\begin{cases}0&x'\in V'\\ g(x')&\text{otherwise.}\end{cases}
\]
Let us verify 
\[
a1_W(x)=\tilde g(\pi(x))\quad\forall x\in X_\pi\setminus\pi^{-1}(L'). 
\tag{$*$}
\]
If $x$ is in $\pi^{-1}(V')$, 
then the both sides of $(*)$ are zero. 
If $x$ is not in $\pi^{-1}(V')$, then 
\[
a1_W(x)=a1_U(x)=g(\pi(x))=\tilde g(\pi(x)), 
\]
that is, $(*)$ holds true. 
In view of Definition \ref{defofreduction}, 
$\Pi(a1_W)$ is given by 
the function $(a1_W)_{L'}\in C_c(X_\pi,A)$. 
We prove $(a1_W)_{L'}=a1_C$. 
If $x$ is in $X_\pi\setminus\pi^{-1}(L')$, then 
both sides are clearly zero. 
If $x$ is in $\pi^{-1}(L')\setminus\pi^{-1}(\pi(C))$, then 
$\pi(x)\in L'\setminus\pi(C)\subset V'$, 
and so $x$ is not in $W$. 
Hence, again, both sides are zero. 
If $x\in C$, then $x\in U$ and $\pi(x)\in\pi(C)$. 
So, $x$ is in $W$, and $(a1_W)_{L'}(x)=a=a1_C(x)$. 
If $x$ is in $\pi^{-1}(\pi(C))\setminus C$, then 
$x$ is not in $U$, 
and hence $(a1_W)_{L'}(x)=0=a1_C(x)$. 
This completes the proof of the surjectivity of $\Pi$. 

Finally, we show that $\Ker\Pi$ equals $C_c(X',A)$. 
Let $f\in C_c(X',A)$. 
For any compact open subset $L'\subset X'_\pi$, 
it is easy to see that 
the function $f_{L'}$ belongs to $C_c(X'_\pi,A)$. 
Thus, $\Pi(f)=0$. 
Take $f\in\Ker\Pi$. 
Let $L'\subset X'_\pi$ and $g:X'_\pi\to A$ be 
as in Lemma \ref{welldefofPi}, 
i.e.\ $f(x)=g(\pi(x))$ holds for all $x\in X_\pi\setminus\pi^{-1}(L')$. 
On the other hand, $f\in\Ker\Pi$ implies $f_{L'}\in C_c(X'_\pi,A)$. 
Therefore, for any $x,y\in X_\pi$ with $\pi(x)=\pi(y)$, 
one has $f(x)=f(y)$. 
Consequently, $f$ belongs to $C_c(X',A)$. 
\end{proof}

Based on the proposition above, 
we may regard the reduction map $\Pi$ as an isomorphism 
from $C_c(X,A)/C_c(X',A)$ to $C_c(X_\pi,A)/C_c(X'_\pi,A)$. 

We conclude this section by showing that 
the reduction maps commute with local homeomorphisms. 

Let $\pi_i:(X_i,\d_i)\to(X'_i,\d'_i)$, 
$\phi:X_1\to X_2$ and $\phi':X'_1\to X'_2$ be 
as in Proposition \ref{regularVSlocalhomeo}. 
Assume that $\pi_1$ and $\pi_2$ are regular. 
Suppose further that 
the metric spaces $X_1,X_2,X'_1,X'_2$ are totally disconnected. 
By Proposition \ref{regularVSlocalhomeo}, 
both $\phi:(X_1)_{\pi_1}\to(X_2)_{\pi_2}$ and 
$\phi':(X'_1)_{\pi_1}\to(X'_2)_{\pi_2}$ are local homeomorphisms. 
Let $A$ be a discrete abelian group. 
From the homomorphisms 
\[
\phi_*:C_c(X_1,A)\to C_c(X_2,A)
\]
and 
\[
\phi'_*:C_c(X'_1,A)\to C_c(X'_2,A), 
\]
we obtain a homomorphism 
\[
\sigma:C_c(X_1,A)/C_c(X'_1,A)\to C_c(X_2,A)/C_c(X'_2,A). 
\]
In the same way, from the homomorphisms 
\[
\phi_*:C_c((X_1)_{\pi_1},A)\to C_c((X_2)_{\pi_2},A)
\]
and 
\[
\phi'_*:C_c((X'_1)_{\pi_1},A)\to C_c((X'_2)_{\pi_2},A), 
\]
we obtain a homomorphism 
\[
\tau:C_c((X_1)_{\pi_1},A)/C_c((X'_1)_{\pi_1},A)
\to C_c((X_2)_{\pi_2},A)/C_c((X'_2)_{\pi_2},A). 
\]

\begin{proposition}\label{reductioncommute}
In the setting above, we have 
\[
\tau\circ\Pi_1=\Pi_2\circ\sigma, 
\]
where the reduction maps $\Pi_i$ are regarded as homomorphisms 
from $C_c(X_i,A)/C_c(X'_i,A)$ 
to $C_c((X_i)_{\pi_i},A)/C_c((X'_i)_{\pi_i},A)$. 
\end{proposition}
\begin{proof}
As in the proof of Proposition \ref{regularVSlocalhomeo}, 
we write $Y_i:=(X_i)_{\pi_i}$ and $Y'_i:=(X'_i)_{\pi_i}$. 

First, we claim the following: 
for any compact subset $C'\subset Y'_2$ and 
any compact subset $D'\subset X'_1$, 
the set $(\phi')^{-1}(C')\cap D'$ is compact in $Y'_1$. 
Let $(y'_k)_k$ be a sequence in $(\phi')^{-1}(C')\cap D'$. 
By passing to a subsequence, 
we may assume that 
$(y'_k)_k$ converges to $y'\in D'$ in $X'_1$ and 
that $(\phi'(y'_k))_k$ converges to $z'\in C'$ in $Y'_2$. 
From the continuity of $\phi'$ and Lemma \ref{finertop}, 
one has $\phi'(y')=z'$. 
In particular, $y'$ belongs to $Y'_1$. 
Since $\phi':X'_1\to X'_2$ is a local homeomorphism, 
we can find an open neighborhood $U'\subset X'_1$ of $y'$ 
such that $\phi':U'\to\phi'(U')$ is a homeomorphism. 
Let $V'\subset Y'_1$ be an open neighborhood of $y'$. 
By Lemma \ref{finertop}, we may assume $V'\subset U'$. 
Then $\phi'(V')\subset Y'_2$ is also an open neighborhood of $\phi'(y')$, 
because $\phi':Y'_1\to Y'_2$ is a local homeomorphism 
by Proposition \ref{regularVSlocalhomeo}. 
It follows that 
for sufficiently large $k$ we have $\phi'(y'_k)\in\phi'(V')$, 
and hence $y'_k\in V'$. 
This means that $(y'_k)_k$ converges to $y'$ in $Y'_1$, 
which implies that $(\phi')^{-1}(C')\cap D'$ is compact in $Y'_1$. 

Take $f\in C_c(X_1,A)$. 
The support $S:=\{x\in X_1\mid f(x)\neq0\}$ of $f$ is 
compact and open in $X_1$. 
The image of $\phi_*(f)\in C_c(X_2,A)$ by the reduction map 
is given by a function $\phi_*(f)_{L'}\in C_c(Y_2,A)$, 
where $L'$ is a suitable compact open subset of $Y'_2$ 
(see Lemma \ref{welldefofPi} and Definition \ref{defofreduction}). 
In other words, 
$\Pi_2\circ\sigma$ sends $f+C_c(X'_1,A)$ 
to $\phi_*(f)_{L'}+C_c(Y'_2,A)$. 
On the other hand, the image of $f$ by the reduction map is 
given by $f_{K'}\in C_c(Y_1,A)$, 
where $K'$ is a suitable compact open subset of $Y'_1$ 
(see Lemma \ref{welldefofPi} and Definition \ref{defofreduction}). 
In other words, 
$\tau\circ\Pi_1$ sends $f+C_c(X'_1,A)$ 
to $\phi_*(f_{K'})+C_c(Y'_2,A)$. 

Choose a compact open subset $C'\subset Y'_2$ and 
a compact open subset $D'\subset X'_1$ so that 
\[
S\subset\pi_1^{-1}(D'),\quad L'\subset C'
\quad\text{and}\quad 
K'\subset(\phi')^{-1}(C')\cap D'. 
\]
Clearly $E':=(\phi')^{-1}(C')\cap D'$ is clopen in $Y'_1$. 
Thanks to the claim we have proved above, 
$E'$ is compact in $Y'_1$. 
As $K'\subset E'$, one has 
\[
\phi_*(f_{K'})+C_c(Y'_2,A)=\phi_*(f_{E'})+C_c(Y'_2,A). 
\]
Similarly, by $L'\subset C'$, 
\[
\phi_*(f)_{L'}+C_c(Y'_2,A)=\phi_*(f)_{C'}+C_c(Y'_2,A). 
\]
Now, for every $y\in Y_2$ we have 
\[
\phi_*(f)_{C'}(y)
=\sum_{\substack{y\in\pi_2^{-1}(C') \\ y=\phi(x)}}f(x)
\]
and 
\[
\phi_*(f_{E'})(y)
=\sum_{x\in\phi^{-1}(y)}f_{E'}(x)
=\sum_{\substack{y=\phi(x) \\ x\in\pi_1^{-1}(E')}}f(x)
=\sum_{\substack{y=\phi(x) \\ x\in\phi^{-1}(\pi_2^{-1}(C'))}}f(x), 
\]
where the last equality follows from 
$E'=(\phi')^{-1}(C')\cap D'$ and $S\subset\pi_1^{-1}(D')$. 
Therefore we obtain $\phi_*(f)_{C'}=\phi_*(f_{E'})$, 
which completes the proof. 
\end{proof}

%%%%%%%%%%%%%%%%%%%%%%%%%%%%%%%%%%%%%%%%%%%%%%%%%%%%%%%%%%%%
\subsection{Factor groupoid situation}
%%%%%%%%%%%%%%%%%%%%%%%%%%%%%%%%%%%%%%%%%%%%%%%%%%%%%%%%%%%%

In this subsection, 
we apply the results obtained in the last subsection 
to \'etale groupoids arising from the factor groupoid situation. 
The precise setting is given as follows. 

\begin{setting}\label{factor}
Let $\G$ and $\G'$ be \'etale groupoids and 
let $\pi:\G\to\G'$ be a continuous homomorphism. 
We assume the following. 
\begin{itemize}
\item $\pi$ is surjective. 
\item $\pi$ is proper. 
\item For any $x\in\G^{(0)}$, $\pi$ induces a bijection 
between $r^{-1}(x)$ to $r^{-1}(\pi(x))$. 
\item Both $\G$ and $\G'$ are totally disconnected. 
\end{itemize}
\end{setting}

\begin{remark}
In \cite[Section 6]{Pu21Munster}, 
the groupoids $\G$ and $\G'$ are not assumed to be 
\'etale or totally disconnected. 
Thus, the results of \cite[Section 6]{Pu21Munster} can be applied to 
actions of continuous groups on connected spaces. 
In this article, we restrict our attention to 
\'etale groupoids on totally disconnected spaces. 
\end{remark}

When $(X_i,\d_i)$, $i=1,2,\dots,n$ are metric spaces, 
we use the $\ell^1$-metric 
for the product space $X_1\times X_2\times\dots\times X_n$. 
That is, the metric of the product space is given by 
\[
\d((x_i)_i,(y_i)_i):=\sum_{i=1}^n\d_i(x_i,y_i). 
\]
The following lemma says that 
it suffices to check the regularity of $\pi:\G^{(0)}\to\G'^{(0)}$ 
in order to deduce 
that of $\pi:\G\to\G'$ or $\pi^{(n)}:\G^{(n)}\to\G'^{(n)}$. 

\begin{lemma}\label{pi^nregular}
Let $\pi:\G\to\G'$ be as in Setting \ref{factor}. 
The following are equivalent. 
\begin{enumerate}
\item $\pi:\G^{(0)}\to\G'^{(0)}$ is regular. 
\item $\pi:\G\to\G'$ is regular. 
\item For some $n\in\N\setminus\{1\}$, 
$\pi^{(n)}:\G^{(n)}\to\G'^{(n)}$ is regular. 
\end{enumerate}
\end{lemma}
\begin{proof}
The range maps $r:\G\to\G^{(0)}$ and $r:\G'\to\G'^{(0)}$ are 
local homeomorphisms, 
and so the equivalence (1)$\iff$(2) follows 
from Proposition \ref{regularVSlocalhomeo}. 
The projection $(g_1,g_2,\dots,g_n)\mapsto g_1$ is 
a local homeomorphism from $\G^{(n)}$ to $\G$, 
The same is true for $\G'^{(n)}\to\G'$. 
Hence the equivalence (2)$\iff$(3) follows 
from Proposition \ref{regularVSlocalhomeo}, too. 
\end{proof}

When $\pi:\G\to\G'$ is as in Setting \ref{factor}, 
it is easy to see that 
$\G_\pi$ and $\G'_\pi$ are subgroupoids of $\G$ and $\G'$, 
respectively. 
Our next task is to show that they are \'etale groupoids 
under the topologies induced by the metrics $\d_\pi$ and $\d'_\pi$. 

\begin{lemma}
Let $\pi:\G\to\G'$ be as in Setting \ref{factor}. 
Suppose that $\pi$ is regular. 
Then 
\[
\mathcal{M}:=\{(g,h')
\mid g\in\G_\pi,\ h'\in\G'_\pi,\ s(\pi(g)){=}r(h')\}
\]
is a closed subset of $\G_\pi\times\G'_\pi$ 
and there exists a continuous map $m:\mathcal{M}\to\G_\pi$ such that 
$\pi(g^{-1}m(g,h'))=h'$ holds for any $(g,h')\in\mathcal{M}$. 
\end{lemma}
\begin{proof}
The set 
\[
\mathcal{M}_0:=\{(g,h')\mid g\in\G,\ h'\in\G',\ s(\pi(g)){=}r(h')\}
\]
is clearly closed in $\G\times\G'$ and 
$\mathcal{M}$ is the intersection of 
$\mathcal{M}_0$ and $\G_\pi\times\G'_\pi$. 
Hence $\mathcal{M}$ is closed in $\G_\pi\times\G'_\pi$. 
Take $(g,h')\in\mathcal{M}$. 
By assumption, there exists a unique element $h\in\G_\pi$ 
such that $\pi(h)=h'$ and $s(g)=r(h)$. 
Define $m(g,h')=gh$. 
Clearly we have $\pi(g^{-1}m(g,h'))=h'$. 
Let us show that $m$ is continuous. 
Suppose that 
a sequence $((g_k,h'_k))_k$ converges to $(g,h')$ in $\mathcal{M}$. 
Let $h_k:=g_k^{-1}m(g_k,h'_k)$. 
We can find a compact open $\G'$-bisection $U\subset\G'$ 
containing $\pi(m(g,h'))$. 
Let $V\subset\G_\pi$ be an open neighborhood of $m(g,h')$ in $\G_\pi$. 
By Lemma \ref{finertop}, we may assume $V\subset\pi^{-1}(U)$. 
By Proposition \ref{consequofregular}, $\pi:\G_\pi\to\G'_\pi$ is open, 
and so $\pi(V)$ is an open neighborhood of $\pi(m(g,h'))=\pi(g)h'$ 
in $\G'_\pi$. 
By Proposition \ref{regularVSlocalhomeo}, 
the range map $r:\G_\pi\to\G_\pi^{(0)}$ is a local homeomorphism. 
Therefore, $r(V)$ is an open neighborhood of $r(m(g,h'))=r(g)$. 
Since $(g_k)_k$ converges to $g$ in $\G_\pi$, 
for sufficiently large $k$, we get $r(g_k)\in r(V)$. 
Thus, there exists $a_k\in V$ such that $r(g_k)=r(a_k)$. 
As the multiplication on $\G'$ is continuous, 
for sufficiently large $k$, one has $\pi(g_kh_k)=\pi(g_k)h'_k\in U$. 
It follows from $\pi(a_k)\in\pi(V)\subset U$ that 
$\pi(a_k)$ equals $\pi(g_kh_k)$, 
because $U$ is a $\G'$-bisection. 
This, together with $r(g_kh_k)=r(g_k)=r(a_k)$, implies $a_k=g_kh_k$. 
Hence, for sufficiently large $k$, $g_kh_k$ is in $V$, 
which means that $m(g_k,h'_k)=g_kh_k$ converges to $m(g,h')$. 
\end{proof}

\begin{proposition}\label{Gpiisetale}
Let $\pi:\G\to\G'$ be as in Setting \ref{factor}. 
Suppose that $\pi$ is regular. 
Then $\G_\pi$ and $\G'_\pi$ are \'etale groupoids. 
\end{proposition}
\begin{proof}
By Proposition \ref{consequofregular}, 
$\G_\pi$ and $\G'_\pi$ are locally compact metric spaces. 
By Proposition \ref{regularVSlocalhomeo}, 
the range maps $r:\G_\pi\to\G_\pi^{(0)}$ and $r:\G'_\pi\to(\G'_\pi)^{(0)}$ 
are local homeomorphisms. 
It remains for us to show that 
the groupoid operations are continuous. 
Let us consider the multiplication on $\G'_\pi$. 
Let $m:\mathcal{M}\to\G_\pi$ be as in the lemma above. 
As $m$ is continuous, it gives rise to a continuous map $m_{\textrm{H}}$ 
from compact subsets of $\mathcal{M}$ to compact subsets of $\G_\pi$. 
For any $(g',h')\in(\G'_\pi)^{(2)}$, it is easy to see that 
\[
m_{\textrm{H}}(\pi^{-1}(g')\times\{h'\})=\pi^{-1}(g'h'). 
\]
Suppose that 
a sequence $((g'_k,h'_k))_k$ converges to $(g',h')$ in $(\G'_\pi)^{(2)}$. 
It is not so hard to see that the inequality 
\[
\dH(\pi^{-1}(g'_k)\times\{h'_k\},\pi^{-1}(g')\times\{h'\})
\leq\dH(\pi^{-1}(g'_k),\pi^{-1}(g'))+\d'(h'_k,h')
\]
holds, where $\dH$ in the left-hand side stands for 
the Hausdorff metric on $\mathcal{M}$ and 
$\dH$ in the right-hand side stands for the Hausdorff metric on $\G$. 
Since $((g'_k,h'_k))_k$ converges to $(g',h')$, we have 
\[
\lim_{k\to\infty}\dH(\pi^{-1}(g'_k),\pi^{-1}(g'))=0
\quad\text{and}\quad 
\lim_{k\to\infty}\d'(h'_k,h')=0, 
\]
and so 
\[
\lim_{k\to\infty}
\dH(\pi^{-1}(g'_k)\times\{h'_k\},\pi^{-1}(g')\times\{h'\})=0. 
\]
From the continuity of $m_{\textrm{H}}$, we obtain 
$\dH(\pi^{-1}(g'_kh'_k),\pi^{-1}(g'h'))\to0$ as $k\to\infty$. 
Hence, $(g'_kh'_k)_k$ converges to $g'h'$ in $\G'_\pi$, as desired. 
The continuity of the inverse operation is obvious. 
Finally, the continuity of the groupoid operations in $\G_\pi$ 
easily follows from the arguments for $\G'_\pi$. 
\end{proof}

Clearly, 
the space $\G_\pi^{(n)}$ of composable strings of $n$ elements 
equals the space $(\G^{(n)})_{\pi^{(n)}}$ 
constructed from the regular map $\pi^{(n)}:\G^{(n)}\to\G'^{(n)}$ 
as sets. 
The following technical lemma ensures 
that their topologies also agree. 

\begin{lemma}
Let $\pi:\G\to\G'$ be as in Setting \ref{factor}. 
Suppose that $\pi$ is regular. 
\begin{enumerate}
\item For any $n\in\N\setminus\{1\}$, 
the spaces $\G_\pi^{(n)}$ and $(\G^{(n)})_{\pi^{(n)}}$ are 
canonically homeomorphic. 
\item For any $n\in\N\setminus\{1\}$, 
the spaces $(\G'_\pi)^{(n)}$ and $(\G'^{(n)})_{\pi^{(n)}}$ are 
canonically homeomorphic. 
\end{enumerate}
\end{lemma}
\begin{proof}
Notice that $\pi^{(n)}$ is regular by Lemma \ref{pi^nregular}. 
We show (1). 
By definition, 
\[
\G_\pi^{(n)}=\{(g_1,g_2,\dots,g_n)\in\G_\pi^n\mid s(g_i)=r(g_{i+1})\}
\]
and 
\[
(\G^{(n)})_{\pi^{(n)}}=\{(g_1,g_2,\dots,g_n)\in\G^{(n)}
\mid\#(\pi^{(n)})^{-1}(g_1,g_2,\dots,g_n)>1\}
\]
agree as sets. 
We have to show that they share the same topology. 
The topology on $(\G^{(n)})_{\pi^{(n)}}$ 
comes from the regularity of $\pi^{(n)}:\G^{(n)}\to\G'^{(n)}$. 
The map 
\[
\phi:(g_1,g_2,\dots,g_n)\mapsto r(g_1)
\]
is a local homeomorphism from $\G^{(n)}$ to $\G^{(0)}$, 
and the map 
\[
\phi':(g'_1,g'_2,\dots,g'_n)\mapsto r(g'_1)
\]
is a local homeomorphism from $\G'^{(n)}$ to $\G'^{(0)}$. 
By applying Proposition \ref{regularVSlocalhomeo} to $\phi$ and $\phi'$, 
we can conclude that 
$\phi:(\G^{(n)})_{\pi^{(n)}}\to\G_\pi^{(0)}$ and 
$\phi':(\G'^{(n)})_{\pi^{(n)}}\to(\G'_\pi)^{(0)}$ are local homeomorphisms. 
On the other hand, 
the space $\G_\pi^{(n)}$ is equipped with the relative topology 
from the product space $\G_\pi^n$, and the map 
\[
\phi:(g_1,g_2,\dots,g_n)\mapsto r(g_1)
\]
gives a local homeomorphism from $\G_\pi^{(n)}$ to $\G_\pi^{(0)}$. 
It follows that 
$\G_\pi^{(n)}$ and $(\G^{(n)})_{\pi^{(n)}}$ share the same topology. 
One can prove (2) in the same way. 
\end{proof}

Now, we are in a position to prove that 
$\G\to\G'$ and $\G_\pi\to\G'_\pi$ have 
isomorphic relative homology groups. 
In the proof, 
the reduction map introduced in Section 4.1 plays a crucial role. 

\begin{proposition}\label{iso_factor}
Let $\pi:\G\to\G'$ be as in Setting \ref{factor}. 
Suppose that $\pi$ is regular. 
Let $A$ be a discrete abelian group. 
We denote by $\Pi_n$ the reduction map 
arising from $\pi^{(n)}:\G^{(n)}\to\G'^{(n)}$. 
\begin{enumerate}
\item $\Pi_n$ is an isomorphism 
from $C_c(\G^{(n)},A)/C_c(\G'^{(n)},A)$ 
to $C_c(\G_\pi^{(n)},A)/C_c((\G'_\pi)^{(n)},A)$. 
\item The family $\Pi_*$ of the reduction maps $(\Pi_n)_n$ gives 
an isomorphism between the chain complexes 
$\mathcal{C}(\G,A)/\mathcal{C}(\G',A)$ and 
$\mathcal{C}(\G_\pi,A)/\mathcal{C}(\G'_\pi,A)$. 
Thus, we have the following diagram: 
\[
\xymatrix@M=10pt{
0 \ar[r] & \mathcal{C}(\G',A) \ar[r]^{\pi^*} & \mathcal{C}(\G,A) \ar[r] 
& \mathcal{C}(\G,A)/\mathcal{C}(\G',A) \ar[r] \ar[d]^{\Pi_*}_\cong & 0 \\
0 \ar[r] & \mathcal{C}(\G'_\pi,A) \ar[r]_{\pi^*} 
& \mathcal{C}(\G_\pi,A) \ar[r] 
& \mathcal{C}(\G_\pi,A)/\mathcal{C}(\G'_\pi,A) \ar[r] & 0. 
}
\]
\end{enumerate}
\end{proposition}
\begin{proof}
(1) is a direct consequence of Proposition \ref{reductionisiso}. 
In order to see (2), 
it suffices to verify that 
the reduction maps $\Pi_n$ commute with differentials, 
and this follows from Proposition \ref{reductioncommute}. 
\end{proof}

We let $H_*(\G/\G',A)$ (resp.\ $H_*(\G_\pi/\G'_\pi,A)$) 
symbolically denote the homology groups of 
the chain complex $\mathcal{C}(\G,A)/\mathcal{C}(\G',A)$ 
(resp.\ $\mathcal{C}(\G_\pi,A)/\mathcal{C}(\G'_\pi,A)$). 

\begin{theorem}\label{LES_factor}
Let $\pi:\G\to\G'$ be as in Setting \ref{factor}. 
Suppose that $\pi$ is regular. 
Let $A$ be a discrete abelian group. 
Then we have the following diagram: 
\[
\xymatrix@M=10pt{
\dots \ar[r] & H_n(\G',A) \ar[r]^{H_n^*(\pi)} & H_n(\G,A) \ar[r] 
& H_n(\G/\G',A) \ar[r] \ar@{=}[d] & \dots \\
\dots \ar[r] & H_n(\G'_\pi,A) \ar[r]^{H_n^*(\pi)} & H_n(\G_\pi,A) \ar[r] 
& H_n(\G_\pi/\G'_\pi,A) \ar[r] & \dots, 
}
\]
where the two horizontal sequences are exact. 
\end{theorem}

%%%%%%%%%%%%%%%%%%%%%%%%%%%%%%%%%%%%%%%%%%%%%%%%%%%%%%%%%%%%
\section{Examples: SFT groupoids}
%%%%%%%%%%%%%%%%%%%%%%%%%%%%%%%%%%%%%%%%%%%%%%%%%%%%%%%%%%%%

%%%%%%%%%%%%%%%%%%%%%%%%%%%%%%%%%%%%%%%%%%%%%%%%%%%%%%%%%%%%
\subsection{SFT groupoids}
%%%%%%%%%%%%%%%%%%%%%%%%%%%%%%%%%%%%%%%%%%%%%%%%%%%%%%%%%%%%

In this subsection, we briefly recall the notion of SFT groupoids. 
The reader may refer to \cite{Ma12PLMS,Ma15crelle} 
for a more complete background. 

Let $(V,E)$ be a finite directed graph, 
where $V$ is a finite set of vertices 
and $E$ is a finite set of edges. 
For $e\in E$, $i_E(e)$ denotes the initial vertex of $e$ and 
$t_E(e)$ denotes the terminal vertex of $e$. 
Let $B:=(B(\xi,\eta))_{\xi,\eta\in V}$ be 
the adjacency matrix of $(V,E)$, that is, 
\[
B(\xi,\eta):=\#\{e\in E\mid i_E(e)=\xi,\ t_E(e)=\eta\}. 
\]
We assume that $B$ is irreducible 
(i.e.\ for all $\xi,\eta\in V$ 
there exists $n\in\N$ such that $B^n(\xi,\eta)>0$) 
and that $B$ is not a permutation matrix. 
Define 
\[
X_{(V,E)}:=\{(x_k)_{k\in\N}\in E^\N
\mid t(x_k)=i(x_{k+1})\quad\forall k\in\N\}. 
\]
With the product topology, $X_{(V,E)}$ is a Cantor set. 
Define a surjective continuous map $\sigma:X_{(V,E)}\to X_{(V,E)}$ by 
\[
\sigma(x)_k:=x_{k+1}\quad k\in\N,\ x=(x_k)_k\in X_{(V,E)}. 
\]
In other words, $\sigma$ is the (one-sided) shift on $X_{(V,E)}$.  
It is easy to see that $\sigma$ is a local homeomorphism. 
The dynamical system $(X_{(V,E)},\sigma)$ is called 
the one-sided irreducible shift of finite type (SFT) 
associated with the graph $(V,E)$ (or the matrix $B$). 

The \'etale groupoid $\G_{(V,E)}$ is given by 
\[
\G_{(V,E)}:=\{(x,n,y)\in X_{(V,E)}\times\Z\times X_{(V,E)}\mid
\exists k,l\in\N,\ n=k{-}l,\ \sigma^k(x)=\sigma^l(y)\}. 
\]
The topology of $\G_{(V,E)}$ is generated by the sets 
\[
\{(x,k{-}l,y)\in\G_{(V,E)}\mid x\in P,\ y\in Q,\ \sigma^k(x)=\sigma^l(y)\}, 
\]
where $P,Q\subset X_{(V,E)}$ are open and $k,l\in\N$. 
Two elements $(x,n,y)$ and $(x',n',y')$ in $\G_{(V,E)}$ are composable 
if and only if $y=x'$, and the multiplication and the inverse are 
\[
(x,n,y)\cdot(y,n',y')=(x,n{+}n',y'),\quad (x,n,y)^{-1}=(y,-n,x). 
\]
We identify $X_{(V,E)}$ with the unit space $\G_{(V,E)}^{(0)}$ 
via $x\mapsto(x,0,x)$. 
We call $\G_{(V,E)}$ the SFT groupoid 
associated with the graph $(V,E)$. 

The homology groups of the SFT groupoid $\G_{(V,E)}$ was 
computed in \cite{Ma12PLMS}. 

\begin{theorem}[{\cite[Theorem 4.14]{Ma12PLMS}}]
Let $(V,E)$, $B$ and $\G_{(V,E)}$ be as above. 
Then one has 
\[
H_n(\G_{(V,E)})\cong\begin{cases}\Coker(I-B^t)&n=0\\
\Ker(I-B^t)&n=1\\0&n\geq2, \end{cases}
\]
where the matrix $B$ acts on the abelian group $\Z^V$ by multiplication. 
\end{theorem}

See \cite{Ma16Adv,FKPS19Munster} for further developments.

%%%%%%%%%%%%%%%%%%%%%%%%%%%%%%%%%%%%%%%%%%%%%%%%%%%%%%%%%%%%
\subsection{Examples}
%%%%%%%%%%%%%%%%%%%%%%%%%%%%%%%%%%%%%%%%%%%%%%%%%%%%%%%%%%%%

\begin{example}\label{exoffactorSFT}
Let $(V,E)$ be a finite directed graph and
let $B:=(B(\xi,\eta))_{\xi,\eta\in V}$ be the adjacency matrix of $(V,E)$. 
We assume that $B$ is irreducible 
and that $B$ is not a permutation matrix. 
We also identify $B$ with the induced homomorphism $\Z^V\to\Z^V$. 

We construct new finite directed graphs $(\V,\E)$ and $(\W,\F)$ 
as follows. 
Set 
\[
\V:=V,\quad \W:=V\times\Z_2
\]
\[
\E:=E\sqcup\W,\quad \F:=\E\times\Z_2. 
\]
Define $i_\E:\E\to\V$ and $t_\E:\E\to\V$ by 
\[
i_\E(e):=i_E(e),\quad i_\E(v,p):=v,\quad 
t_\E(e):=t_E(e),\quad t_\E(v,p):=v
\]
for every $e\in E$, $v\in V$ and $p\in\Z_2$. 
Then, the adjacency matrix of $(\V,\E)$ is $B+2I$, 
where $I$ denotes the identity matrix. 
We define $i_\F:\F\to\W$ and $t_\F:\F\to\W$ by 
\[
i_\F(e,q):=(i_\E(e),q),\quad i_\F(v,p,q):=(v,q)
\]
\[
t_\F(e,q):=(t_\E(e),q),\quad t_\F(v,p,q):=(v,p{+}q). 
\]
for every $e\in E$, $v\in V$ and $p,q\in\Z_2$. 
Then, the adjacency matrix of $(\W,\F)$ can be written as 
\[
\begin{bmatrix}
B+I & I \\ I & B+I
\end{bmatrix}. 
\]

Now, by sending $(e,q)\in\F$ to $e\in\E$ and 
$(v,p,q)\in\F$ to $(v,q)\in\E$, 
we obtain a graph homomorphism $\rho:(\W,\F)\to(\V,\E)$. 
We remark that $\rho$ has the following property: 
for any $\omega\in\W$, 
the map $\rho:t_{\F}^{-1}(\omega)\to t_{\E}^{-1}(\rho(\omega))$ 
is bijective. 
(This property is known as ``left-covering''. 
Some relevant arguments can be found in 
\cite[Proposition 1]{BKM85PAMS} and \cite[Corollary 3.9]{Pu00JLMS}.) 
Let $\pi:\G_{(\W,\F)}\to\G_{(\V,\E)}$ be the groupoid homomorphism 
induced by $\rho$. 
Put $\G:=\G_{(\W,\F)}$ and $\G':=\G_{(\V,\E)}$. 
It is not so hard to see that $\pi$ is surjective and proper. 
Moreover, we have 
\[
(\G'_\pi)^{(0)}=\left\{(x_k)_k\in X_{(\V,\E)}\mid x_k\in E
\text{ for all sufficiently large $k$}\right\}, 
\]
and $\pi$ is two to one on $(\G'_\pi)^{(0)}$. 
By using Lemma \ref{pi^nregular}, one can verify that $\pi$ is regular. 
Define a compact open subset $Y\subset(\G'_\pi)^{(0)}$ by 
\[
Y:=\left\{(x_k)_k\in(\G'_\pi)^{(0)}
\mid x_k\in E\quad\forall k\in\N\right\}. 
\]
Then $Y$ is full and $\G'_\pi|Y$ is isomorphic to $\G_{(V,E)}$. 
Thus the \'etale groupoid $\G'_\pi$ is Kakutani equivalent to $\G_{(V,E)}$, 
and so 
\[
H_n(\G'_\pi)\cong\begin{cases}\Coker(I-B^t)&n=0\\
\Ker(I-B^t)&n=1\\0&n\geq2. \end{cases}
\]
The \'etale groupoid $\G_\pi$ is isomorphic to 
the direct product of two copies of $\G'_\pi$, and 
the homomorphisms 
$H^*_n(\pi):H_n(\G'_\pi)\to H_n(\G_\pi)=H_n(\G'_\pi)\oplus H_n(\G'_\pi)$ 
are given by $c\mapsto(c,c)$. 
Hence we get $H_n(\G_\pi/\G'_\pi)\cong H_n(\G'_\pi)$, 
where $H_n(\G_\pi/\G'_\pi)$ is as in Theorem \ref{LES_factor}. 
Therefore, Theorem \ref{LES_factor} yields the following exact sequence: 
\[
%\xymatrix@M=10pt{
\xymatrix{
0 \ar[r] & \Ker(B^t+I) \ar[r] & 
{\Ker\begin{bmatrix}B^t&I\\I&B^t\end{bmatrix}} \ar[r] & \Ker(B^t-I) \ar[d] \\
0 & \Coker(B^t-I) \ar[l] & 
{\Coker\begin{bmatrix}B^t&I\\I&B^t\end{bmatrix}} \ar[l] & 
\Coker(B^t+I) \ar[l]
}
\]
%\[
%%\xymatrix@M=10pt{
%\xymatrix{
%0 \ar[r] & \Ker(B^t+I) \ar[r] & 
%{\Ker\begin{bmatrix}B^t&I\\I&B^t\end{bmatrix}} \ar[r] & \Ker(B^t-I) \\
%\ar[r] & \Coker(B^t+I) \ar[r] & 
%{\Coker\begin{bmatrix}B^t&I\\I&B^t\end{bmatrix}} \ar[r] & 
%\Coker(B^t-I) \ar[r] & 0. 
%}
%\]
\end{example}

\begin{example}
Let $(V,E)$, $(\V,\E)$ and $(\W,\F)$ be as in Example \ref{exoffactorSFT}. 
We consider a graph homomorphism $\rho:(\W,\F)\to(\V,\E)$ 
different from the previous one. 
Namely, we associate $(e,q)\in\F$ with $e\in\E$ and 
$(v,p,q)\in\F$ with $(v,p{+}q)\in\E$. 
We remark that $\rho$ has the following property: 
for any $\omega\in\W$, 
the map $\rho:i_{\F}^{-1}(\omega)\to i_{\E}^{-1}(\rho(\omega))$ 
is bijective. 
(This property is known as ``right-covering''. 
Compare with the map $\rho$ appearing in the previous example.)
Let $\tilde\rho:X_{(\W,\F)}\to X_{(\V,\E)}$ denote 
the induced continuous map between the infinite path spaces. 
It is easy to see that $\tilde\rho$ is surjective and two to one. 
We define $\tau:\F\to\Z_2$ by $\tau(\ep,q):=q$ 
for $\ep\in\E$ and $q\in\Z_2$. 
Then 
\[
(f_n)_n\mapsto\left(\tau(f_1),\tilde\rho((f_n)_n)\right)
\]
defines a homeomorphism $X_{(\W,\F)}\to\Z_2\times X_{(\V,\E)}$, 
which gives rise to a continuous injective homomorphism 
$\iota:\G_{(\W,\F)}\to\mathcal{K}\times\G_{(\V,\E)}$, 
where $\mathcal{K}:=\Z_2\times\Z_2$ is the trivial groupoid on $\Z_2$. 
Put $\G:=\mathcal{K}\times\G_{(\V,\E)}$ and $\G':=\G_{(\W,\F)}$. 
Identifying $\G'$ with $\iota(\G')$, 
we think of $\G'$ as a subgroupoid of $\G$. 
Clearly $\G'$ is open in $\G$, and 
\[
r(\G\setminus\G')=\Z_2\times\left\{(x_k)_k\in X_{(\V,\E)}\mid x_k\in E
\text{ for all sufficiently large $k$}\right\}. 
\]
One can verify that the inclusion $\G'\subset\G$ is regular. 
By Theorem \ref{Hisetale}, 
$\H:=\G|r(\G\setminus\G')$ and $\H':=\G'|r(\G\setminus\G')$ become 
\'etale groupoids. 
Define a compact open subset $Y\subset\H^{(0)}=r(\G\setminus\G')$ by 
\[
Y:=\{0\}\times\left\{(x_k)_k\in X_{(\V,\E)}
\mid x_k\in E\quad\forall k\in\N\right\}. 
\]
Then $Y$ is full and $\H|Y$ is isomorphic to $\G_{(V,E)}$, 
i.e. the \'etale groupoid $\H$ is Kakutani equivalent to $\G_{(V,E)}$. 
Also, 
$\H'$ is Kakutani equivalent to the product of two copies of $\G_{(V,E)}$. 
The homomorphisms $H_n(\H')\to H_n(\H)$ 
induced by the inclusion $\H'\subset\H$ are given 
by $(a,b)\mapsto a+b$. 
Hence we get $H_n(\H/\H')\cong H_{n-1}(\H)$, 
where $H_n(\H/\H')$ is as in Theorem \ref{LES_sub}. 
Therefore, Theorem \ref{LES_sub} yields the following exact sequence: 
\[
%\xymatrix@M=10pt{
\xymatrix{
0 \ar[r] & \Ker(B^t-I) \ar[r] & 
{\Ker\begin{bmatrix}B^t&I\\I&B^t\end{bmatrix}} \ar[r] & \Ker(B^t+I) \ar[d] \\
0 & \Coker(B^t+I) \ar[l] & 
{\Coker\begin{bmatrix}B^t&I\\I&B^t\end{bmatrix}} \ar[l] & 
\Coker(B^t-I) \ar[l] 
}
\]
%\[
%%\xymatrix@M=10pt{
%\xymatrix{
%0 \ar[r] & \Ker(B^t-I) \ar[r] & 
%{\Ker\begin{bmatrix}B^t&I\\I&B^t\end{bmatrix}} \ar[r] & \Ker(B^t+I) \\
%\ar[r] & \Coker(B^t-I) \ar[r] & 
%{\Coker\begin{bmatrix}B^t&I\\I&B^t\end{bmatrix}} \ar[r] & 
%\Coker(B^t+I) \ar[r] & 0. 
%}
%\]
\end{example}

%%%%%%%%%%%%%%%%%%%%%%%%%%%%%%%%%%%%%%%%%%%%%%%%%%%%%%%%%%%%
\section{Examples: hyperplane groupoids}
%%%%%%%%%%%%%%%%%%%%%%%%%%%%%%%%%%%%%%%%%%%%%%%%%%%%%%%%%%%%

%%%%%%%%%%%%%%%%%%%%%%%%%%%%%%%%%%%%%%%%%%%%%%%%%%%%%%%%%%%%
\subsection{Hyperplane groupoids}
%%%%%%%%%%%%%%%%%%%%%%%%%%%%%%%%%%%%%%%%%%%%%%%%%%%%%%%%%%%%

In this subsection, we recall the notion of hyperplane systems 
from \cite{Pu10CMP} and introduce hyperplane groupoids. 

When $P\subset\R^N$ is an affine hyperplane of co-dimension one, 
it divides the space $\R^N$ into two closed half-spaces 
whose intersection is $P$. 
By an orientation of $P$, 
we mean that we have a fixed choice of labeling these as $P^0$ and $P^1$. 

\begin{definition}[{\cite[Definition 3.1]{Pu10CMP}}]
A hyperplane system is a triple $\mathcal{A}=(\R^N,L,\mathcal{P})$, 
where $L\subset\R^N$ is a finitely generated subgroup and 
$\mathcal{P}$ is a countable collection of 
co-dimension one oriented affine hyperplanes in $\R^N$ 
which is invariant under the action of $L$. 
That is, for each $P$ in $\mathcal{P}$ and $s$ in $L$, 
the translation $P+s$ is also in $\mathcal{P}$ 
and $(P+s)^i=P^i+s$ for $i=0,1$. 
\end{definition}

\begin{theorem}[{\cite[Section 3]{Pu10CMP}}]\label{hypergroupoid}
From a hyperplane system $(\R^N,L,\mathcal{P})$, 
we can construct a locally compact metrizable space $X$, 
an action $\phi:L\curvearrowright X$ and 
a continuous proper surjection $q:X\to\R^N$ satisfying the following. 
\begin{enumerate}
\item $q:X\to\R^N$ is a factor map, 
i.e.\ $q(\phi_s(x))=q(x)+s$ holds for all $x\in X$ and $s\in L$. 
\item For every $P\in\mathcal{P}$ and $i=0,1$, 
we let $[P^i]\subset X$denote the closure of $q^{-1}(P^i\setminus P)$. 
Then $[P^0]$ and $[P^1]$ do not have intersection, 
and their union equals $X$. 
\item The collection 
\[
\left\{[P^i]\mid P\in\mathcal{P},\ i=0,1\right\}
\cup\left\{q^{-1}(U)\mid\text{$U\subset\R^N$ is open in $\R^N$}\right\}
\]
is a subbasis of the topology of $X$. 
\end{enumerate}
\end{theorem}

We would like to briefly describe how to construct the space $X$ 
in the theorem above.
Let $C_0(\R^N)$ denote the $C^*$-algebra of 
continuous functions on $\R^N$ vanishing at infinity. 
We regard $C_0(\R^N)$ as a subalgebra of 
the $C^*$-algebra $L^\infty(\R^N)$ consisting of 
essentially bounded measurable functions on $\R^N$. 
Let $\mathfrak{A}$ be the (abelian) $C^*$-algebra generated by 
\[
C_0(\R^N)\cup
\left\{f1_{P^0}\mid f\in C_0(\R^N),\ P\in\mathcal{P}\right\}
\]
in $L^\infty(\R^N)$. 
Then the space $X$ is defined as the Gelfand spectrum of $\mathfrak{A}$, 
that is, $\mathfrak{A}\cong C_0(X)$. 
Since $\mathfrak{A}$ contains $C_0(\R^N)$, 
there exists a continuous proper surjection $q:X\to\R^N$. 

\begin{definition}
Let $(\R^N,L,\mathcal{P})$ be a hyperplane system and 
let $\phi:L\curvearrowright X$ be as in the theorem above. 
We call the transformation groupoid $L\ltimes X$ 
the hyperplane groupoid arising from $(\R^N,L,\mathcal{P})$. 
\end{definition}

Notice that the space $X$ need not be totally disconnected. 
In fact, when $\mathcal{P}$ is empty, it is just $\R^N$. 
In the following argument, however, 
we are interested only in the case that $X$ is totally disconnected. 

The following is an example of hyperplane systems on $\R$. 
It is known that such one-dimensional systems are related to 
tiling spaces on $\R$. 
See \cite[Example 46]{FHK02CMP} for instances. 

\begin{example}\label{Denjoy}
%We give a basic example of hyperplane systems. 
Let $L\subset\R$ be the subgroup generated by $1$ and $\sqrt{2}$. 
Set 
\[
\mathcal{P}:=\left\{\{s\}\mid s\in L\right\}, 
\]
and $\{s\}^0:=(-\infty,s]$ and $\{s\}^1:=[s,\infty)$. 
Notice that a co-dimension one hyperplane of $\R$ is just a point. 
By the theorem above, 
we obtain $\phi:L\curvearrowright X$ and $q:X\to\R$. 
The quotient space of $X$ by the $\Z$-action generated by $\phi_1$ 
is a Cantor set, on which $\phi_{\sqrt{2}}$ naturally acts. 
In other words, 
the restriction of the hyperplane groupoid $\G:=L\ltimes X$ 
to $[\{0\}^1]\cap[\{1\}^0]$ is isomorphic to a Cantor minimal system 
arising from a Denjoy homeomorphism on the circle 
whose rotation number is $\sqrt{2}-1$. 
(Here, a homeomorphism on the circle is called a Denjoy homeomorphism 
when it has an irrational rotation number and 
is not conjugate to a rotation. 
Such a homeomorphism has a unique closed invariant set, 
which is homeomorphic to a Cantor set. 
We refer the reader to \cite{PSS86JOP} for the details.) 
So, $\G$ is Kakutani equivalent to a Cantor minimal system 
associated with a Denjoy homeomorphism. 
%In other words, 
%the restriction of the hyperplane groupoid $\G:=L\ltimes X$ 
%to $[\{0\}^1]\cap[\{1\}^0]$ is isomorphic to a Denjoy system 
%(the restriction of a Denjoy homeomorphism on the circle 
%to its minimal closed subset, see \cite{PSS86JOP}), 
%and so $\G$ is Kakutani equivalent to a Denjoy system. 
In particular, one has 
\[
H_n(\G)\cong\begin{cases}\Z^2&n=0\\\Z&n=1\\0&n\geq2. \end{cases}
\]
We also remark that this Kakutani equivalence implies 
the strong Morita equivalence between the $C^*$-algebras 
$C_0(X)\rtimes L$ and $C_0(X/L_1)\rtimes L_{\sqrt{2}}$, 
where $L_1:=\langle 1\rangle$ and 
$L_{\sqrt{2}}:=\langle\sqrt{2}\rangle$ 
(see \cite[Remark 3.10, 3.12]{Pu10CMP}). 
By using the notation of Theorem \ref{hypergroupoid} (2), 
the two generators of $H_0(\G)$ are given by the clopen subsets 
\[
[\{0\}^1]\cap[\{1\}^0]
\quad\text{and}\quad[\{0\}^1]\cap[\{\sqrt{2}\}^0]. 
\]
We denote the equivalence classes of them in $H_0(\G)$ 
by $\alpha$ and $\beta$. 
Thus, $H_0(\G)=\Z\alpha+\Z\beta$. 
\end{example}

%%%%%%%%%%%%%%%%%%%%%%%%%%%%%%%%%%%%%%%%%%%%%%%%%%%%%%%%%%%%
\subsection{Example: the octagonal tiling}
%%%%%%%%%%%%%%%%%%%%%%%%%%%%%%%%%%%%%%%%%%%%%%%%%%%%%%%%%%%%

In this subsection, we compute the homology groups of 
the hyperplane groupoid corresponding to the octagonal tiling. 
This computation is parallel to 
the arguments in \cite[Section 6]{Pu10CMP}, 
in which the $K$-groups of the $C^*$-algebras of 
the octagonal tiling are computed. 

First, 
we define a hyperplane system $\mathcal{A}=(\R^2,L,\mathcal{P})$ as follows. 
For $j=0,1,2,3$, we let $e_j:=(\cos j\pi/4,\sin j\pi/4)\in\R^2$ and 
let $L\subset\R^2$ be the abelian group generated by $(e_j)_j$. 
The group $L$ is isomorphic to $\Z^4$ and dense in $\R^2$. 
For $j=0,1,2,3$, 
we let $P_j\subset\R^2$ be the one-dimensional subspace spanned by $e_j$, 
and give the orientation as in Figure 1. 
\begin{figure}
\centering
\includegraphics[pagebox=cropbox,clip]{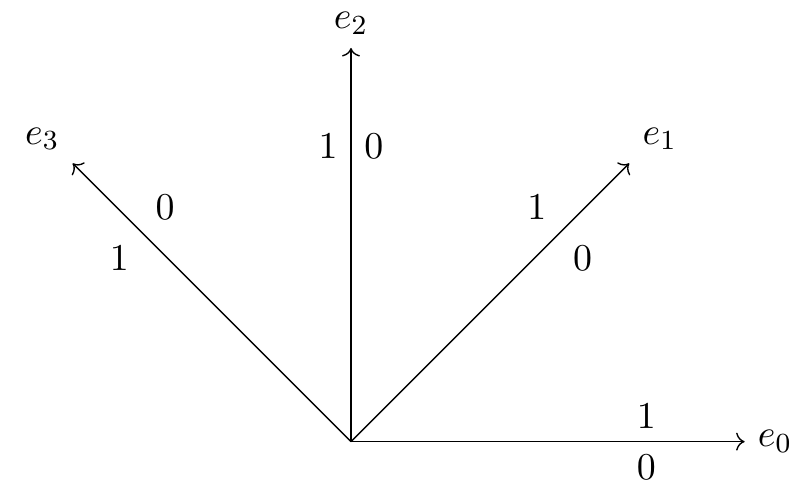}
\caption{the orientation of $P_j$}
\end{figure}

Set 
\[
\mathcal{P}_j:=\{P_j+s\mid s\in L\}, 
\]
and give the orientation for $P_j+s$ in $\mathcal{P}_j$ 
by $(P_j+s)^i=P_j^i+s$, $i=0,1$. 
Finally, we put $\mathcal{P}:=\bigcup_{j=0}^3\mathcal{P}_j$ and 
obtain the hyperplane system $\mathcal{A}:=(\R^2,L,\mathcal{P})$. 
Let $\G=L\ltimes X$ be the hyperplane groupoid arising from $\mathcal{A}$ 
and let $q:X\to\R^2$ be the factor map. 
This hyperplane system is equivalent to the octagonal tiling 
in an appropriate sense. 
See \cite{FHK02Memoirs,FHK02CMP,Pu10CMP} and the references therein. 
Our aim in this subsection is to compute the homology groups of $\G$. 

To this end, we need two hyperplane systems 
\[
\mathcal{A}_0:=(\R^2,L,\mathcal{P}_1\cup\mathcal{P}_2)
\]
and 
\[
\mathcal{A}_1:=(\R^2,L,\mathcal{P}_1\cup\mathcal{P}_2\cup\mathcal{P}_3). 
\]
The homology groups of the hyperplane groupoid of $\mathcal{A}_0$ 
are easily computed (Step 1). 
Applying the results in Section 4, 
we determine the homology groups of $\mathcal{A}_1$ in Step 2. 
Then, by using the results in Section 4 again, 
we compute the homology groups $H_n(\G)$ in Step 3. 
This is our strategy. 

For $i=0,1$, by Theorem \ref{hypergroupoid}, 
we obtain an action $\phi_i:L\curvearrowright X_i$ and 
a factor map $q_i:X_i\to\R^2$ such that the diagram below is commutative. 
Let $\G_i:=L\ltimes X_i$ be the hyperplane groupoid. 
\[
\xymatrix@M=10pt{
X \ar[rrd]_-q \ar[r] & X_1 \ar[rd]^-{q_1} \ar[r] & X_0 \ar[d]^-{q_0} \\
& & \R^2
}
\]

\paragraph{Step 1}

First, we look at $\mathcal{A}_0=(\R^2,L,\mathcal{P}_1\cup\mathcal{P}_2)$ 
and compute $H_n(\G_0)$. 
The abelian group $L\cong\Z^4$ is freely generated 
by $e_1$, $e_2$, $e_0{+}e_2$ and $e_1{+}e_3$. 
Consider the quotient space $\tilde X_0$ of $X_0$ 
by the $\Z^2$-action induced by $e_1,e_2$. 
Then, as in Example \ref{Denjoy}, 
$\tilde X_0$ is a Cantor set 
on which $\langle e_0{+}e_2,e_1{+}e_3\rangle\cong\Z^2$ acts. 
It is easy to see that 
this dynamical system is conjugate to the product of 
two copies of the Denjoy system described in Example \ref{Denjoy}. 
Therefore, one gets 
\[
H_n(\G_0)\cong\begin{cases}\Z^4&n=0\\\Z^4&n=1\\\Z&n=2\\0&n\geq3. \end{cases}
\]
By using the notation of Example \ref{Denjoy}, 
the generators of $H_0(\G_0)=\Z^2\otimes\Z^2$ can be written by 
$\alpha\otimes\alpha$, $\alpha\otimes\beta$, $\beta\otimes\alpha$ 
and $\beta\otimes\beta$.

\paragraph{Step 2}

Next, we discuss the hyperplane system 
$\mathcal{A}_1=(\R^2,L,\mathcal{P}_1\cup\mathcal{P}_2\cup\mathcal{P}_3)$. 
There exists a factor map $p:X_1\to X_0$, 
which gives rise to a continuous homomorphism $\pi:\G_1\to\G_0$. 
Obviously $\pi$ is surjective and proper. 
For any $y\in X_0$, one has $\#p^{-1}(y)$ is $1$ or $2$. 
When $\#p^{-1}(y)=2$, 
there exists a unique $P\in\mathcal{P}_3$ such that 
$p^{-1}(y)=\{x_0,x_1\}$ and $x_i\in[P^i]$ for $i=0,1$. 
Suppose that a sequence $(y_n)_n$ converges to $y$ in $X_0$. 
Then, for any $P'\in\mathcal{P}_3\setminus\{P\}$, we can see that 
$p^{-1}(y_n)$ and $x_i$ are the same side of $P'$ eventually. 
This implies that $p:X_1\to X_0$ is regular, 
and so is $\pi:\G_1\to\G_0$ by Lemma \ref{pi^nregular}. 
The reader may refer to \cite[Proposition 3.5]{Pu10CMP} 
for the topologies of $X_0$ and $X_1$ 
(see also the maps $i^j_P$ introduced in \cite[Section 4]{Pu10CMP}). 
Now, we would like to apply the results of Section 4 for $\pi$. 
Let us look at the space $(\G_0)_\pi^{(0)}$. 
Under the identification of $\G_0^{(0)}$ with $X_0$, 
this space can be written as 
\[
Y_0:=\{x\in X_0\mid \#p^{-1}(x)=2\}. 
\]
Clearly $q_0(Y_0)\subset\R^2$ is equal to 
\[
\bigcup_{P\in\mathcal{P}_3}P=\bigcup_{s\in L}(P_3+s), 
\]
and $Y_0$ (with the topology described in Definition \ref{defofregular}) 
is homeomorphic to the product of $Y_0\cap q_0^{-1}(P_3)$ and 
a countable discrete set. 
Put $\theta:=1/\sqrt{2}$ and 
\[
C:=\{l+m\theta\in\R\mid l,m\in\Z\}. 
\]
When $z$ is in $P_3\setminus\{ce_3\mid c\in C\}$, 
$q_0^{-1}(z)$ is a singleton and belongs to $Y_0$. 
When $z$ is in $\{ce_3\mid c\in C\}$, 
there exists $s_j\in L$ such that $z\in P_j+s_j$ 
for each $j=1,2$, 
and $q_0^{-1}(z)$ consists of four points 
$x_{0,0},x_{0,1},x_{1,0},x_{1,1}$ 
satisfying $x_{k,l}\in[(P_1+s_1)^k]\cap[(P_2+s_2)^l]$. 
Among these points, $x_{0,0}$ and $x_{1,1}$ are in $Y_0$, 
and the others are not. 

Now, the map $p:X_1\to X_0$ is two-to-one on $Y_0$, 
and the restriction of the \'etale groupoid $(\G_0)_\pi$ 
to $Y_0\cap q_0^{-1}(P_3)$ is isomorphic 
to the hyperplane groupoid arising from the hyperplane system 
\[
\mathcal{B}:=(\R,M,\mathcal{Q}), 
\]
where $M:=\langle1,2\theta\rangle\cong\Z^2$ and 
$\mathcal{Q}:=\{\{c\}\mid c\in C\}$. 
Accordingly, $(\G_0)_\pi$ is Kakutani equivalent 
to the hyperplane groupoid of $\mathcal{B}$. 
We remark that 
this procedure of constructing $\mathcal{B}$ is called 
the reduction of hyperplane systems in \cite[Section 4]{Pu10CMP}. 
In a similar fashion to Example \ref{Denjoy}, 
one has 
\[
H_n((\G_0)_\pi)\cong\begin{cases}\Z^3&n=0\\\Z&n=1\\0&n\geq2. \end{cases}
\]
Note that $C=\Z+\Z\theta$ splits into two orbits under the action of $M$, 
namely $\Z+2\Z\theta$ and $\Z+(2\Z{+}1)\theta$. 
Hence $H_0$ of $\mathcal{B}$ is of rank three, 
and the three generators are given by the clopen subsets 
\[
[\{0\}^1]\cap[\{1\}^0],\quad [\{0\}^1]\cap[\{\theta\}^0]
\quad\text{and}\quad[\{0\}^1]\cap[\{2\theta\}^0]. 
\]
The \'etale groupoid $(\G_1)_\pi$ is 
the product of two copies of $(\G_0)_\pi$ 
and the homomorphisms $H_n((\G_0)_\pi)\to H_n((\G_1)_\pi)$ 
are given by $a\mapsto(a,a)$. 
Therefore, Theorem \ref{LES_factor} yields the following exact sequence: 
\[
%\xymatrix@M=10pt{
\xymatrix{
0 \ar[r] & H_2(\G_0) \ar[r] & 
H_2(\G_1) \ar[r] & 0 \ar[d] \\
 & \Z \ar[d]_-\partial & 
H_1(\G_1) \ar[l] & H_1(\G_0) \ar[l] \\
 & H_0(\G_0) \ar[r] & 
H_0(\G_1) \ar[r] & 
\Z^3 \ar[r] & 0. 
}
\]
%\[
%%\xymatrix@M=10pt{
%\xymatrix{
%0 \ar[r] & H_2(\G_0) \ar[r] & 
%H_2(\G_1) \ar[r] & 0 \\
%\ar[r] & H_1(\G_0) \ar[r] & 
%H_1(\G_1) \ar[r] & \Z \\
%\ar[r]^-\partial & H_0(\G_0) \ar[r] & 
%H_0(\G_1) \ar[r] & 
%\Z^3 \ar[r] & 0. 
%}
%\]
In order to determine $H_n(\G_1)$, 
it suffices to know the map $\partial:\Z\to H_0(\G_0)\cong\Z^4$. 
Roughly speaking, 
the generator $a$ of $H_1((\G_0)_\pi)\cong\Z$ is given by 
the $2\theta$ translation of the unit interval modulo $\Z$. 
See Figure 2. 
\begin{figure}
\centering
\medskip
\includegraphics[pagebox=cropbox,clip]{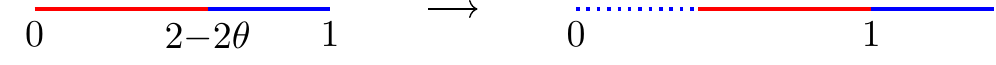}
\caption{$2\theta$ translation modulo $\Z$}
\end{figure}
Then, its `lift' to $C_c(\G_1,\Z)$ is 
the $(2\theta{-}1)e_3$ translation of the triangle 
described in Figure 3. 
\begin{figure}[h]
\centering
\medskip
\includegraphics[pagebox=cropbox,clip]{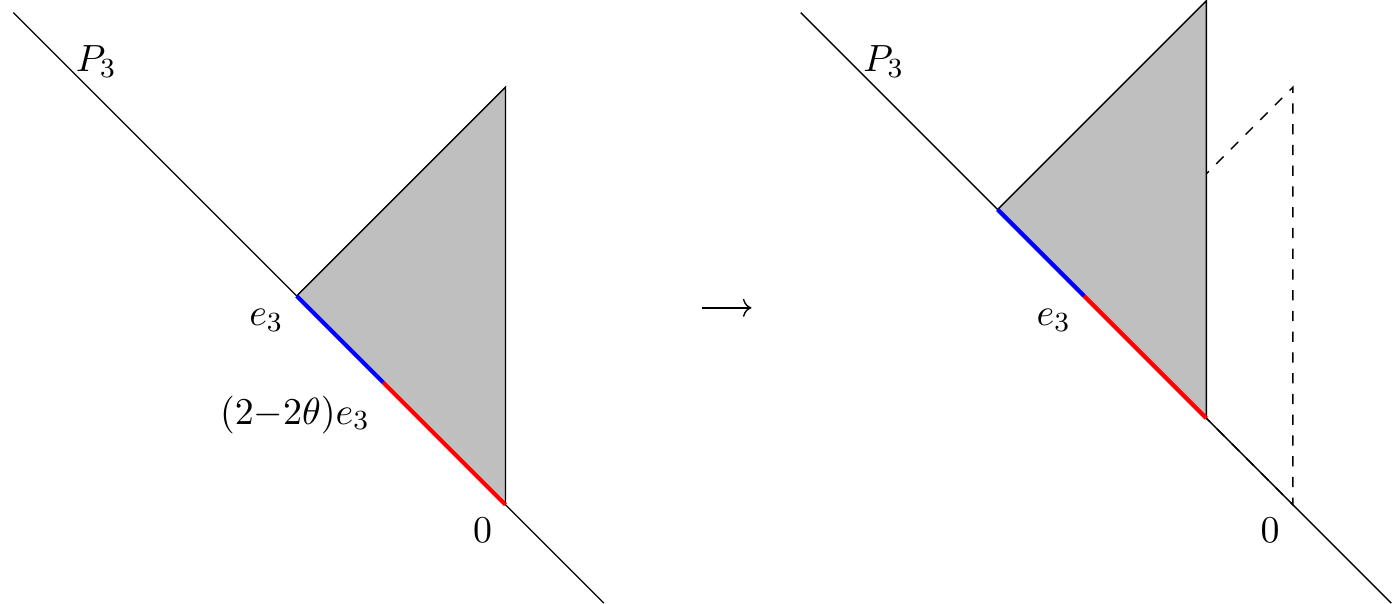}
\caption{$(2\theta{-}1)e_3$ translation of triangle}
\end{figure}
Here, the gray triangle in the left side of Figure 3 represents 
a clopen subset of $\G_1^{(0)}$, which can be written as 
\[
[P_3^0]\cap[P_2^1]\cap[P_1^0+e_3]
\]
by using the notation of Theorem \ref{hypergroupoid} (2). 
Now, $\partial(a)\in H_0(\G_0)$ corresponds to 
the `difference' of the two gray triangles. 
Figure 4 illustrates it. 
\begin{figure}[h]
\centering
\medskip
\includegraphics[pagebox=cropbox,clip]{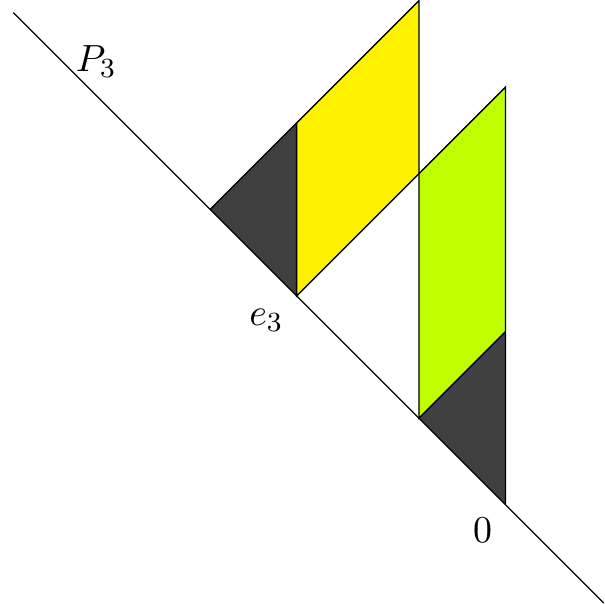}
\caption{`difference' of two triangles}
\end{figure}
The two black triangles cancel each other by the $e_3$ translation. 
Hence $\partial(a)\in H_0(\G_0)$ is equal to 
the equivalence class of the yellow parallelogram 
minus that of the green parallelogram. 
Notice that all the sides of these parallelograms are 
parallel to either $P_1$ or $P_2$, 
which means that they are well-defined as clopen subsets of $\G_0^{(0)}$. 
The yellow one is actually a rhombus, whose sides are $2-2\theta$. 
It follows that 
the equivalence class in $H_0(\G)$ of the yellow rhombus is 
\begin{align*}
& \left(2\alpha-\beta\right)\otimes\left(2\alpha-\beta\right)\\
&=4\left(\alpha\otimes\alpha\right)-2\left(\alpha\otimes\beta\right)
-2\left(\beta\otimes\alpha\right)+\beta\otimes\beta. 
\end{align*}
The sides of the green one are $-1+2\theta$ and $2(-1+2\theta)$, 
and so its equivalence class is 
\begin{align*}
& \left(-\alpha+\beta\right)\otimes\left(-2\alpha+2\beta\right)\\
&=2\left(\alpha\otimes\alpha\right)-2\left(\alpha\otimes\beta\right)
-2\left(\beta\otimes\alpha\right)
+2\left(\beta\otimes\beta\right). 
\end{align*}
Consequently, 
\[
\partial(a)=2\left(\alpha\otimes\alpha\right)-\beta\otimes\beta, 
\]
which implies that $H_0(\G_0)/\Im\partial$ is isomorphic to $\Z^3$. 
This, together with the computation in Step 1, yields the following: 
\[
H_n(\G_1)\cong\begin{cases}\Z^6&n=0\\\Z^4&n=1\\\Z&n=2\\0&n\geq3. \end{cases}
\]

\paragraph{Step 3}

Finally, we discuss the hyperplane system 
$\mathcal{A}=(\R^2,L,\mathcal{P})$. 
There exists a factor map $p:X\to X_1$, 
which gives rise to a continuous homomorphism $\pi:\G\to\G_1$. 
It is routine to check that $\pi$ is surjective, proper and regular. 
We would like to apply the results of Section 4 for $\pi$. 
Let $\mathcal{B}=(\R,M,\mathcal{Q})$ be as in Step 2. 
In the same way as Step 2, 
one can see that 
the \'etale groupoid $(\G_1)_\pi$ is Kakutani equivalent 
to the hyperplane groupoid arising from $\mathcal{B}$. 
Furthermore, the \'etale groupoid $\G_\pi$ is 
the product of two copies of $(\G_1)_\pi$ 
and the homomorphisms $H_n((\G_1)_\pi)\to H_n(\G_\pi)$ 
are given by $a\mapsto(a,a)$. 
Therefore, Theorem \ref{LES_factor} yields the following exact sequence: 
\[
%\xymatrix@M=10pt{
\xymatrix{
0 \ar[r] & H_2(\G_1) \ar[r] & 
H_2(\G) \ar[r] & 0 \ar[d] \\
 & \Z \ar[d]_-\partial & 
H_1(\G) \ar[l] & H_1(\G_1) \ar[l] \\
 & H_0(\G_1) \ar[r] & 
H_0(\G) \ar[r] & 
\Z^3 \ar[r] & 0. 
}
\]
%\[
%%\xymatrix@M=10pt{
%\xymatrix{
%0 \ar[r] & H_2(\G_1) \ar[r] & 
%H_2(\G) \ar[r] & 0 \\
%\ar[r] & H_1(\G_1) \ar[r] & 
%H_1(\G) \ar[r] & \Z \\
%\ar[r]^-\partial & H_0(\G_1) \ar[r] & 
%H_0(\G) \ar[r] & 
%\Z^3 \ar[r] & 0. 
%}
%\]
In order to determine $H_n(\G)$, 
it suffices to know the map $\partial:\Z\to H_0(\G_1)$. 
In the same way as Step 2, 
we let $a\in H_1((\G_1)_\pi)\cong\Z$ be the generator. 
Consider the clopen subset 
\[
[P_0^1]\cap[P_1^0]\cap[P_3^1+e_0]
\]
of $\G^{(0)}=X$, which corresponds to the triangle 
with vertices $(0,0)$, $(1,0)$ and $(1/2,1/2)$. 
We look at the $(2\theta{-}1)e_0$ translation of this triangle. 
As shown in Figure 5, 
the two black triangles cancel each other by the $e_0$ translation. 
\begin{figure}
\centering
\includegraphics[pagebox=cropbox,clip]{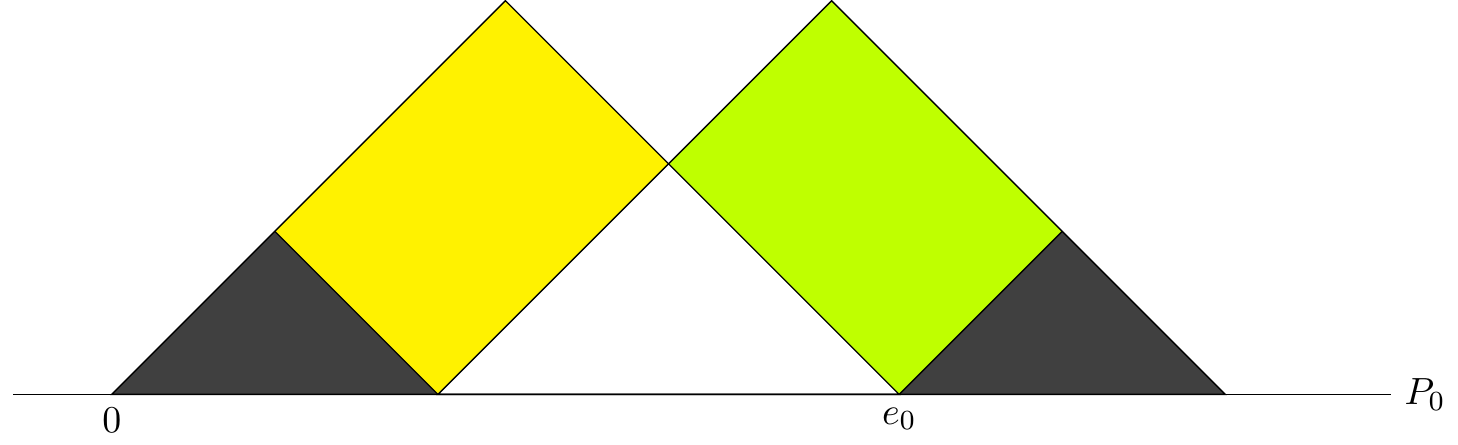}
\caption{`difference' of two triangles}
\end{figure}
Hence $\partial(a)\in H_0(\G_1)$ is equal to 
the equivalence class of the yellow rectangle 
minus that of the green rectangle. 
Now, this element is zero in $H_0(\G_1)$, 
which can be seen from Figure 6. 
\begin{figure}
\centering
\includegraphics[pagebox=cropbox,clip]{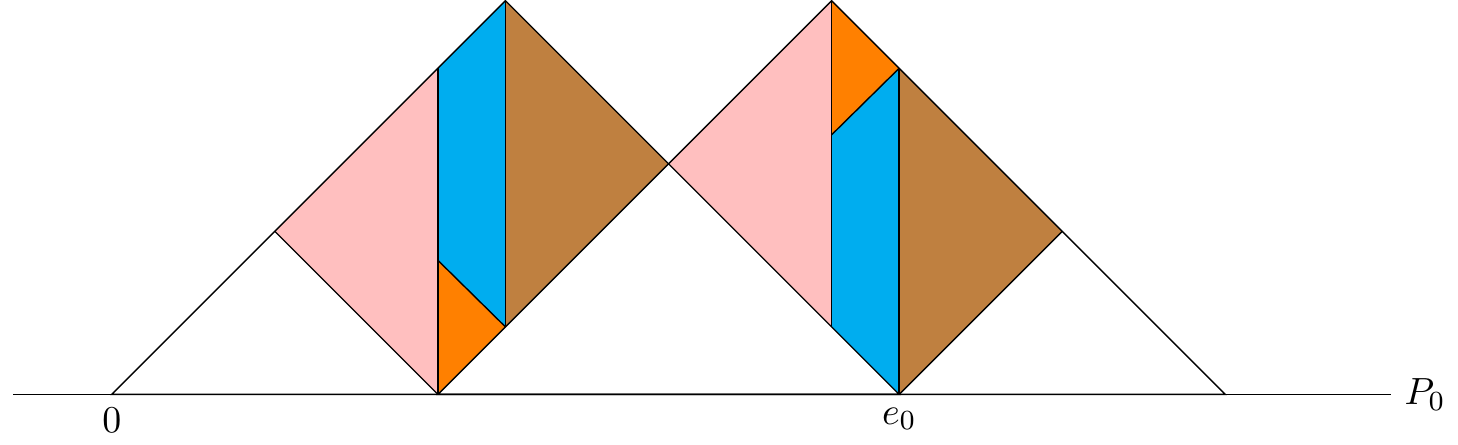}
\caption{a puzzle}
\end{figure}
This, together with the computation in Step 2, yields the following: 
\[
H_n(\G)\cong\begin{cases}\Z^9&n=0\\\Z^5&n=1\\\Z&n=2\\0&n\geq3. \end{cases}
\]
This computation result agrees with that given in 
\cite[Example 6.4]{FHK02CMP}, \cite[Section 10]{KP00CRM}.

%%%%%%%%%%%%%%%%%%%%%%%%%%%%%%%%%%%%%%%%%%%%%%%%%%%%%%%%%%%%
\subsection{Example: the Penrose tiling}
%%%%%%%%%%%%%%%%%%%%%%%%%%%%%%%%%%%%%%%%%%%%%%%%%%%%%%%%%%%%

In this subsection, we compute the homology groups of 
the hyperplane groupoid corresponding to the Penrose tiling. 

First, 
we define a hyperplane system $\mathcal{A}=(\R^2,L,\mathcal{P})$ as follows. 
For $j=0,1,2,3,4$, we let $e_j:=(\cos j\pi/5,\sin j\pi/5)\in\R^2$ and 
let $L\subset\R^2$ be the abelian group generated by $(e_j)_j$. 
Notice that $e_0-e_1+e_2-e_3+e_4=0$. 
The group $L$ is isomorphic to $\Z^4$ and dense in $\R^2$. 
For $j=0,1,2,3,4$, 
we let $P_j\subset\R^2$ be the one-dimensional subspace spanned by $e_j$, 
and give the orientation as in Figure 7. 
\begin{figure}[h]
\centering
\medskip
\includegraphics[pagebox=cropbox,clip]{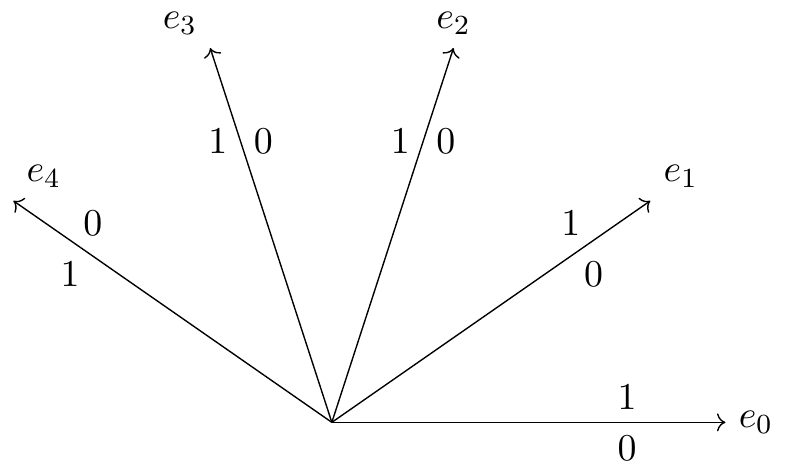}
\caption{the orientation of $P_j$}
\end{figure}

In the same way as Section 6.2, we set 
\[
\mathcal{P}_j:=\{P_j+s\mid s\in L\}, 
\]
and give the orientation for $P_j+s$ in $\mathcal{P}_j$ 
by $(P_j+s)^i=P_j^i+s$, $i=0,1$. 
Finally, we put $\mathcal{P}:=\bigcup_{j=0}^4\mathcal{P}_j$ and 
obtain the hyperplane system $\mathcal{A}:=(\R^2,L,\mathcal{P})$. 
Let $\G=L\ltimes X$ be the hyperplane groupoid arising from $\mathcal{A}$ 
and let $q:X\to\R^2$ be the factor map. 
This hyperplane system is equivalent to the Penrose tiling 
in an appropriate sense. 
See \cite{FHK02Memoirs,FHK02CMP,Pu10CMP} and the references therein. 
Our aim in this subsection is to compute the homology groups of $\G$. 

To this end, we need two hyperplane systems 
\[
\mathcal{A}_0:=(\R^2,L,\mathcal{P}_1\cup\mathcal{P}_3)
\]
and 
\[
\mathcal{A}_1:=(\R^2,L,
\mathcal{P}_1\cup\mathcal{P}_2\cup\mathcal{P}_3\cup\mathcal{P}_4). 
\]
In the same way as the octagonal tiling, 
by using the results in Section 4 twice, 
we will compute the homology groups. 
For $i=0,1$, by Theorem \ref{hypergroupoid}, 
we obtain an action $\phi_i:L\curvearrowright X_i$ and 
a factor map $q_i:X_i\to\R^2$. 
Let $\G_i:=L\ltimes X_i$ be the hyperplane groupoid. 

Put $\theta:=(\sqrt{5}-1)/2$. 
Define a one-dimensional hyperplane system $\mathcal{B}:=(\R,M,Q)$ 
by $M:=\langle1,\theta\rangle$ and $Q:=\{\{c\}\mid c\in M\}$. 
Let $\H$ be the hyperplane groupoid of $\mathcal{B}$. 
In the same way as Example \ref{Denjoy}, we have 
\[
H_n(\H)\cong\begin{cases}\Z^2&n=0\\\Z&n=1\\0&n\geq2. \end{cases}
\]
By using the notation of Theorem \ref{hypergroupoid} (2), 
the two generators of $H_0(\H)$ are given by the clopen subsets 
\[
[\{0\}^1]\cap[\{1\}^0]
\quad\text{and}\quad[\{0\}^1]\cap[\{\theta\}^0]. 
\]
We denote the equivalence classes of them in $H_0(\H)$ 
by $\alpha$ and $\beta$. 
Thus, $H_0(\H)=\Z\alpha+\Z\beta$.

\paragraph{Step 1}

First, we look at $\mathcal{A}_0=(\R^2,L,\mathcal{P}_1\cup\mathcal{P}_3)$ 
and compute $H_n(\G_0)$. 
The abelian group $L\cong\Z^4$ is freely generated 
by $e_1$, $e_3$, $e_0{+}e_2$ and $e_2{+}e_4$, 
and $e_0{+}e_2=(1{+}\theta)e_1$ and $e_2{+}e_4=(1{+}\theta)e_3$. 
Then, it is not so hard to see that 
$\G_0$ is isomorphic to $\H\otimes\H$. 
Therefore, one gets 
\[
H_n(\G_0)\cong\begin{cases}\Z^4&n=0\\\Z^4&n=1\\\Z&n=2\\0&n\geq3, \end{cases}
\]
and the generators of $H_0(\G_0)=\Z^2\otimes\Z^2$ can be written by 
$\alpha\otimes\alpha$, $\alpha\otimes\beta$, $\beta\otimes\alpha$ 
and $\beta\otimes\beta$.

\paragraph{Step 2}

Next, we discuss the hyperplane system 
$\mathcal{A}_1
=(\R^2,L,\mathcal{P}_1\cup\mathcal{P}_2
\cup\mathcal{P}_3\cup\mathcal{P}_4)$. 
There exists a factor map $p:X_1\to X_0$, 
which gives rise to a continuous homomorphism $\pi:\G_1\to\G_0$. 
In exactly the same way as the Octagonal tiling, 
one can check that $\pi$ is surjective, proper and regular. 
We would like to apply the results of Section 4 for $\pi$. 

The \'etale groupoid $(\G_0)_\pi$ is Kakutani equivalent to $\H\oplus\H$, 
where the first direct summand corresponds to the line $P_2$ and 
the second corresponds to $P_4$. 
The map $p:X_1\to X_0$ is two-to-one on the unit space of $(\G_0)_\pi$, 
and $(\G_1)_\pi$ is Kakutani equivalent to 
$\H\oplus\H\oplus\H\oplus\H$. 
The homomorphisms $H_n((\G_0)_\pi)\to H_n((\G_1)_\pi)$ 
are given by $(a,b)\mapsto(a,a,b,b)$. 
Therefore, Theorem \ref{LES_factor} yields the following exact sequence: 
\[
%\xymatrix@M=10pt{
\xymatrix{
0 \ar[r] & H_2(\G_0) \ar[r] & 
H_2(\G_1) \ar[r] & 0 \ar[d] \\
 & \Z^2 \ar[d]_-\partial & 
H_1(\G_1) \ar[l] & H_1(\G_0) \ar[l] \\
 & H_0(\G_0) \ar[r] & 
H_0(\G_1) \ar[r] & 
\Z^4 \ar[r] & 0. 
}
\]
%\[
%%\xymatrix@M=10pt{
%\xymatrix{
%0 \ar[r] & H_2(\G_0) \ar[r] & 
%H_2(\G_1) \ar[r] & 0 \\
%\ar[r] & H_1(\G_0) \ar[r] & 
%H_1(\G_1) \ar[r] & \Z^2 \\
%\ar[r]^-\partial & H_0(\G_0) \ar[r] & 
%H_0(\G_1) \ar[r] & 
%\Z^4 \ar[r] & 0. 
%}
%\]
In order to determine $H_n(\G_1)$, 
it suffices to know the map $\partial:\Z^2\to H_0(\G_0)\cong\Z^4$. 
Roughly speaking, 
the generator $a$ of $H_1(\H)\cong\Z$ is given by 
the $\theta$ translation of the unit interval modulo $\Z$. 
In the same way as the previous subsection, 
the images of $(a,0)$ and $(0,a)$ under $\partial$ are 
given by `difference' of triangles, 
which is illustrated in Figure 8. 
\begin{figure}
\centering
\includegraphics[pagebox=cropbox,clip]{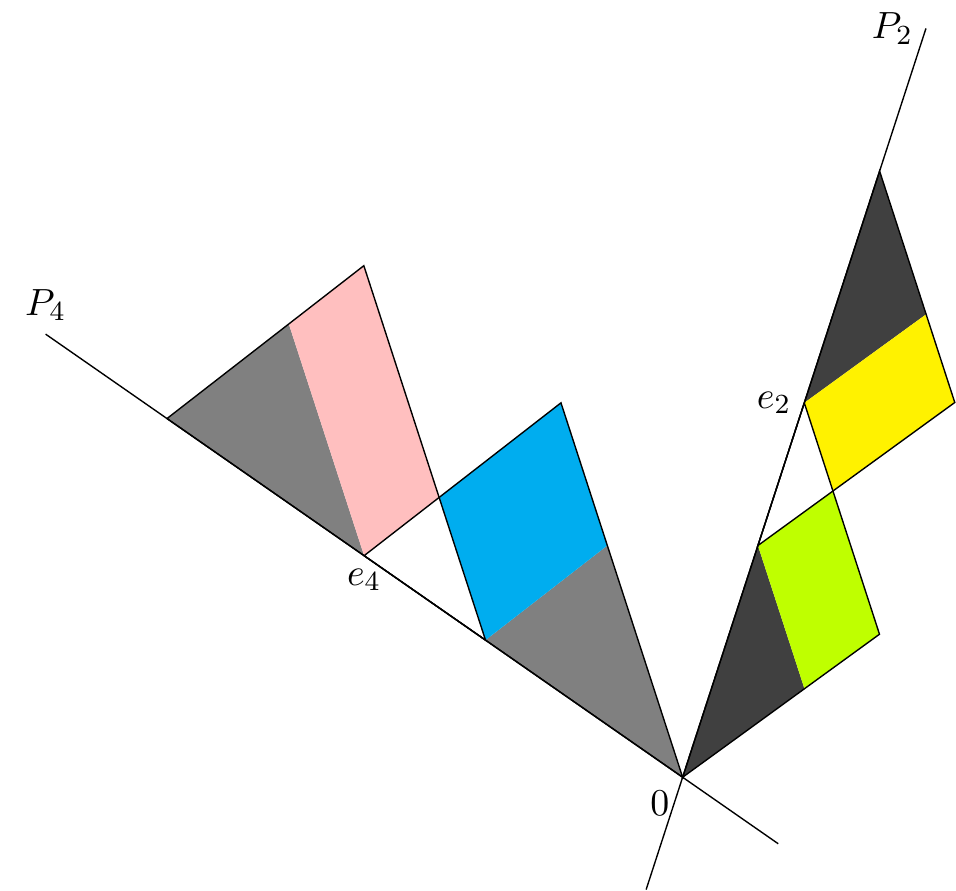}
\caption{`difference' of triangles}
\end{figure}
By looking at the sides of the parallelograms, one gets 
\begin{align*}
\partial(a,0)
&=(\alpha-\beta)\otimes(-\alpha+2\beta)
-(-\alpha+2\beta)\otimes(\alpha-\beta)\\
&=\alpha\otimes\beta-\beta\otimes\alpha
\end{align*}
and 
\begin{align*}
\partial(0,a)
&=(-\alpha+2\beta)\otimes\beta
-(\alpha-\beta)\otimes(\alpha-\beta)\\
&=-\alpha\otimes\alpha+\beta\otimes\alpha+\beta\otimes\beta. 
\end{align*}
Hence the cokernel of $\partial$ is $\Z^2$. 
This, together with the computation in Step 1, yields the following: 
\[
H_n(\G_1)\cong\begin{cases}\Z^6&n=0\\\Z^4&n=1\\\Z&n=2\\0&n\geq3. \end{cases}
\]

\paragraph{Step 3}

Finally, we discuss the hyperplane system 
$\mathcal{A}=(\R^2,L,\mathcal{P})$. 
There exists a factor map $p:X\to X_1$, 
which gives rise to a continuous homomorphism $\pi:\G\to\G_1$. 
It is routine to check that $\pi$ is surjective, proper and regular. 
We would like to apply the results of Section 4 for $\pi$. 
The \'etale groupoid $(\G_1)_\pi$ is Kakutani equivalent to $\H$. 
The map $p:X\to X_1$ is two-to-one on the unit space of $(\G_1)_\pi$, and so 
$(\G_1)_\pi$ is Kakutani equivalent to $\H\oplus\H$. 
The homomorphisms $H_n((\G_1)_\pi)\to H_n(\G)$ is given by $c\mapsto(c,c)$. 
Therefore, Theorem \ref{LES_factor} yields the following exact sequence: 
\[
%\xymatrix@M=10pt{
\xymatrix{
0 \ar[r] & H_2(\G_1) \ar[r] & 
H_2(\G) \ar[r] & 0 \ar[d] \\
 & \Z \ar[d]_-\partial & 
H_1(\G) \ar[l] & H_1(\G_1) \ar[l] \\
 & H_0(\G_1) \ar[r] & 
H_0(\G) \ar[r] & 
\Z^2 \ar[r] & 0. 
}
\]
%\[
%%\xymatrix@M=10pt{
%\xymatrix{
%0 \ar[r] & H_2(\G_1) \ar[r] & 
%H_2(\G) \ar[r] & 0 \\
%\ar[r] & H_1(\G_1) \ar[r] & 
%H_1(\G) \ar[r] & \Z \\
%\ar[r]^-\partial & H_0(\G_1) \ar[r] & 
%H_0(\G) \ar[r] & 
%\Z^2 \ar[r] & 0. 
%}
%\]
Let us show that $\partial:\Z\to H_0(\G_1)$ is trivial. 
Once this is done, we can determine $H_n(\G)$. 
The image of the generator of $H_0(\H)\cong\Z$ under $\partial$ 
is described in Figure 9. 
\begin{figure}
\centering
\includegraphics[pagebox=cropbox,clip]{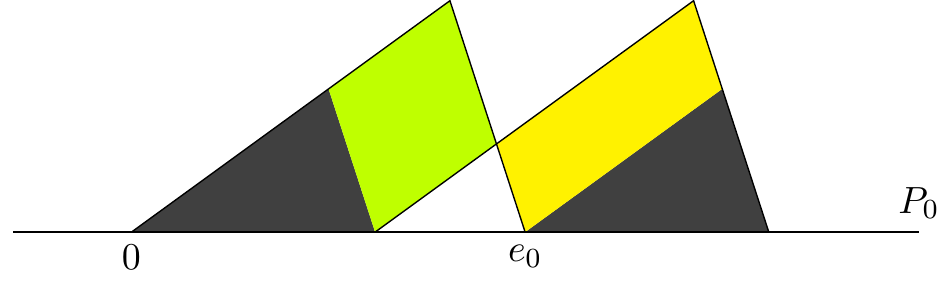}
\caption{`difference' of triangles}
\end{figure}
We would like to show that the equivalence class of 
the yellow parallelogram minus the green parallelogram is zero in $H_0(\G_1)$. 
This can be done in a similar fashion to Figure 6, 
by solving a puzzle. 
However, we take a different approach. 
Let us examine the equivalence class of that element in $H_0(\G_0)$. 
By looking at the sides of the parallelograms, one gets 
\begin{align*}
& \beta\otimes(-\alpha+2\beta)
-(\alpha-\beta)\otimes(\alpha-\beta)\\
&=-\alpha\otimes\alpha+\alpha\otimes\beta+\beta\otimes\beta. 
\end{align*}
This is equal to $\partial(a,0)+\partial(0,a)$ that appeared in Step 2, 
and so its class in $H_0(\G_1)$ is zero. 
Consequently, we have 
\[
H_n(\G_1)\cong\begin{cases}\Z^8&n=0\\\Z^5&n=1\\\Z&n=2\\0&n\geq3. \end{cases}
\]
This computation result agrees with that given 
in \cite[Section 10.4]{AP98ETDS}.

%%%%%%%%%%%%%%%%%%%%%%%%%%%%%%%%%%%%%%%%%%%%%%%%%%%%%%%%%%%%
\section*{Acknowledgement}
%%%%%%%%%%%%%%%%%%%%%%%%%%%%%%%%%%%%%%%%%%%%%%%%%%%%%%%%%%%%

The author thanks the anonymous referees 
for their very detailed reading and 
for giving several valuable comments 
that significantly helped to improve the previous version of this paper.

\newcommand{\noopsort}[1]{}

\end{document}